\documentclass[12pt,a4paper, uplatex]{amsart}

\usepackage{amsmath, amsthm, amsfonts, amssymb,color,geometry,enumerate,stmaryrd}
\usepackage{fancybox}
\usepackage{cases}
\usepackage{fancyhdr}
\usepackage[dvipdfmx]{graphicx}
\usepackage{ascmac}

\theoremstyle{definition}
\newtheorem{theorem}{Theorem}[section]
\newtheorem{proposition}[theorem]{Proposition}
\newtheorem{lemma}[theorem]{Lemma}
\newtheorem{definition}[theorem]{Definition}
\newtheorem{remark}[theorem]{Remark}

\newtheorem{example}[theorem]{Example}

\newtheorem{corollary}[theorem]{Corollary}

\makeatletter

\@addtoreset{equation}{section}
\makeatother

\newcommand{\R}{\mathbb{R}}   
\newcommand{\C}{\mathbb{C}}   
\newcommand{\N}{\mathbb{N}}   
\newcommand{\Z}{\mathbb{Z}}    
\newcommand{\calP}{\mathcal{P}}
\newcommand{\monP}{\mathbb{P}_{\mathrm mon}}
\DeclareMathOperator{\lead}{lead}
\newcommand{\SUB}{\mathbb{D}}

\begin{document}

\title{New combinatorial identity for the set of partitions and limit theorems in finite free probability theory}
\author{Octavio Arizmendi, Katsunori Fujie and Yuki Ueda}
\date{}

\address{
Octavio Arizmendi:
Centro de Investigaci\'{o}n en Matem\'{a}ticas, Guanajuato, Mexico}
\email{octavius@cimat.mx}

\address{
Katsunori Fujie\footnote{All aspects of the project were accomplished during his time at Hokkaido university.}:
Department of Mathematics, Kyoto University, Kyoto, Japan}
\email{fujie.katsunori.42m@st.kyoto-u.ac.jp}

\address{
Yuki Ueda:
Department of Mathematics, Hokkaido University of Education, Hokkaido, Japan}
\email{ueda.yuki@a.hokkyodai.ac.jp}

\maketitle

\begin{abstract}
We provide a refined combinatorial identity for the set of partitions of $\{1,\dots, n\}$, which plays an important role in investigating several limit theorems related to finite free convolutions. 
Firstly, we present the finite free analogue of Sakuma and Yoshida's limit theorem. That is, we provide the limit of $\{D_{1/m}((p_d^{\boxtimes_dm})^{\boxplus_dm})\}_{m\in \N}$ as $m\rightarrow\infty$ in two cases: (i) $m/d\rightarrow t$ for some $t>0$, or (ii) $m/d\rightarrow0$.
The second application presents a central limit theorem for finite free multiplicative convolution. We establish a connection between this theorem and the multiplicative free semicircular distributions through combinatorial identities.
Our last result gives alternative proofs for Kabluchko's limit theorems concerning the unitary Hermite and the Laguerre polynomials.
\end{abstract}

\section{Introduction}

In the 1980s, \textit{free probability theory} was initiated by Voiculescu to attack problems for free products of operator algebras. One of the most important concepts in this theory is the notion of \textit{free independence}, a concept that has found relevant applications to operator algebras and random matrices (see, e.g., \cite{MS, NS, VDN92} and references therein).

Recently, \textit{finite free probability} was introduced by Marcus, Spielman and Srivastava \cite{MSS22} in their study of various consequences of the theory of interlacing polynomials.
As the name suggests, this theory provides a link between polynomial convolutions, free probability theory (see \cite{Marcus}) and random matrices (see \cite{MSS22}).
Two important concepts in this theory are the finite free, additive and multiplicative, convolutions of polynomials.

There has been particularly great progress concerning limit theorems for finite free convolutions and combinatorial structures on finite free probability (see, e.g., \cite{AGVP, AP,  K21, K22, Marcus}).
This article contributes new insights on limit theorems for finite free convolutions and their connection to free probability theory.
In the way to prove these theorems, we have discovered a series of combinatorial identities on sums over partitions, which may be of independent interest.
In particular, these identities allow giving new proofs of the recent results by Kabluchko \cite{ K21, K22} using purely combinatorial tools. 

\subsection{Notation}
Let $\C[x]$ be the set of all polynomials with complex coefficients.
For a polynomial $f\in \C[x]$ expressed as $f(x)=a_dx^d+\cdots +a_1x+a_0 \; (a_d \not = 0)$, we define $\deg (f):=d$ and ${\rm lead}(f):=a_d$.
A polynomial $f\in \C[x]$ is said to be {\it monic} if ${\rm lead}(f)=1$.
Also, $f\in \C[x]$ is said to be {\it real-rooted} if all roots of $f$ are in $\R$. The following subsets of $\C[x]$ are often used in this paper:
\begin{itemize}
\item $\C[x]_0:=\{f\in \C[x] \ |\ f(0)=0\}$;
\item $\mathbb{P}_{\mathrm{mon}}(d):=\{p\in \mathbb{C}[x]\ |\ p \text{ is monic and }\deg (p)=d \}$.
\end{itemize}
In this article, we use the following notations for a polynomial $p$ of degree $d$.
\begin{itemize}
  \item {
     For $p \in \C[x]$ expressed as 
  \[
  p(x)=\sum_{i=0}^d a_i x^{d-i},
  \]
  it is often beneficial to use the notation
  \[
  \widetilde{a}_i(p):= (-1)^i \binom{d}{i}^{-1} a_i, \qquad i=0,1,\dots,d,
  \]
  and then $p$ can be represented by
  \begin{align}\label{eq:poly}
    p(x)=\sum_{i=0}^d (-1)^i \binom{d}{i} \widetilde{a}_i(p) x^{d-i}.
  \end{align}
  }
  \item {The empirical root distribution of $p(x) = \prod_{k=1}^d (x-\lambda_k)$ is a probability measure
  \[
    \mu\llbracket p \rrbracket := \frac{1}{d} \sum_{k=1}^d \delta_{\lambda_k}.
  \]
  }
  \item For $c\neq 0$, the dilation of $p\in \mathbb{P}_{\mathrm{mon}}(d)$ is defined as $D_c(p)(x):=c^dp(x/c)$.
  \item For a monic polynomial $p(x)=\prod_{k=1}^d (x-\lambda_k)$ with nonnegative roots,
  \begin{align}\label{eq:phi_a}
  \phi_\alpha (p)(x):=\prod_{k=1}^d (x-\lambda_k^\alpha), \qquad \alpha>0.
  \end{align}
  \end{itemize}

\subsection{Main results}

Let us present the main results of this paper.
First, we derive a new combinatorial identity for sums over partitions.


\begin{theorem}\label{main1}
Let us consider polynomials $f_1,\dots,f_k$ such that $f_i\in \C[x]_0$ with $\deg (f_i)=m_i$ for $i=1,\dots,k$. Then
\begin{equation}
  \sum_{\pi \in \calP(n)} \prod_{i=1}^k \left(\sum_{V\in \pi} f_i(|V|)\right)\mu_n(\pi,1_n)
   = \begin{cases}
  (n-1)! n^{k-1} \prod_{i=1}^k m_i \text{lead}(f_i), & n=\sum_{i=1}^km_i-(k-1),\\
  0, & n>\sum_{i=1}^km_i-(k-1),
  \end{cases}
  \end{equation}
where $\calP(n)$ denotes the set of all partitions of $[n]$ and $1_n:=\{\{1,\dots, n\}\}\in \calP(n)$, and $\mu_n$ is the M\"{o}bius function on $\calP(n)$, see Section 2.1 for details.
\end{theorem}
{This formula plays a crucial role in investigating the three limit theorems (Theorems \ref{main3}, \ref{main5} and \ref{main4}) outlined below, which establish connections between finite free probability and free probability.}




{Next, we present the finite free analogue of Sakuma and Yoshida's result \cite{SY}, a limit theorem that involves finite free multiplicative and additive convolutions. A detailed description of the findings in \cite{SY} is necessary to articulate the problem.}
Let $\mu$ be a probability measure on $[0,\infty)$ that has the second moment and is not $\delta_0$.
Put $s:=1/m_1(\mu)>0$ and $\alpha=\text{Var}(\mu)/ (m_1(\mu))^2$. Then, Sakuma and Yoshida in \cite[Theorems 9 and 11]{SY} proved that there exists a probability measure $\eta_\alpha$ on $[0,\infty)$, such that
\begin{align}\label{eq:SYlimit}
  D_{s^m/m}((\mu^{\boxtimes m})^{\boxplus m}) \xrightarrow{w} \eta_\alpha,
\end{align}
where $D_c(\rho)(B):=\rho(cB)$ for a probability measure $\rho$ on $\R$, $c\neq0 $ and a Borel set $B$ in $\R$, and also $\xrightarrow{w}$ means the weak convergence.
In addition, it holds for the measure $\eta_\alpha$ that
\[
  \kappa_n(\eta_\alpha)=\frac{(\alpha n)^{n-1}}{n!}, \qquad n\in \N,
\]
where $\kappa_n(\rho)$ is the $n$-th free cumulant of a probability measure $\rho$ on $\R$.
Free cumulants are an important combinatorial tool to treat the free additive and multiplicative convolutions (see \cite{NS} for details).

{As stated in \cite{AP}, it is possible to define the finite free cumulants $\{\kappa_n^{(d)}(p)\}_{n=1}^d$ for $p \in \monP(d)$.
This approach allows treating finite free additive convolution $\boxplus_d$ from a combinatorial perspective.
The definition and detailed facts regarding finite free cumulants are outlined in Section 2.2.}
{Here, we are interested in the limit of the expression given by
\begin{align}\label{eq:Kappa_n(SY)}
\kappa_n^{(d)}\left(D_{1/m}\left((p_d^{\boxtimes_dm})^{\boxplus_dm}\right) \right).
\end{align}
However, it is observed that the limit of \eqref{eq:Kappa_n(SY)} converges to zero as $m\to \infty$ when $d$ is fixed (see Section 4 for details).
To obtain non-trivial limits, 
we consider two regimes under the assumption $\mu\llbracket p_d \rrbracket \xrightarrow{w} \mu$ as $d\to \infty$: (i) $m/d\to t$ for some $t>0$ and (ii) $m/d\to 0$.
Notably, case (i) yields a non-trivial limit for \eqref{eq:Kappa_n(SY)}, but it does not align with the free cumulant of Sakuma and Yoshida's limit distribution (Proposition \ref{thm:m=td}).
The more significant case is (ii), where the limit coincides with the desired free cumulant sequence. We now present the second main result.}


\begin{theorem}\label{main3}
Let us consider $p_d\in \mathbb{P}_{\mathrm{mon}}(d)$ with nonnegative roots such that $\kappa_1^{(d)}(p_d)=1$, and let $\mu$ be a probability measure with compact support. Assume that $\mu\llbracket p_d \rrbracket\xrightarrow{w} \mu$ as $d\rightarrow\infty$. For $n\in \N$, we have
\begin{align*}
\lim_{\substack{d,m\rightarrow \infty\\ m/d\rightarrow 0}}\kappa_n^{(d)} \left(D_{1/m}\left((p_d^{\boxtimes_dm})^{\boxplus_dm}\right)\right) =\frac{(\kappa_2(\mu)n)^{n-1}}{n!},
\end{align*}
i.e. the limit coincides with the $n$-th free cumulant of $\eta_{\kappa_2(\mu)}$.

\end{theorem}

 {As our third main result, we establish a central limit theorem for finite free multiplicative convolution of polynomials with nonnegative roots. Additionally, we explore its connection to free probability through the utilization of the first combinatorial identities (Theorem \ref{main1}).}
\begin{theorem}\label{main5}
\begin{enumerate}[\rm (1)]
\item (Central Limit Theorem) Let $d \ge 2$. Suppose $p(x)=\prod_{k=1}^d (x-e^{\theta_k})$ such that $\frac{1}{d}\sum_{k=1}^d \theta_k=0$ and $\frac{1}{d}\sum_{k=1}^d \theta_k^2 =\sigma^2$. Then we have
\begin{align*}
\lim_{m\rightarrow \infty} \phi_{1/\sqrt{m}}(p)^{\boxtimes_d m} = I_d\left(x; \frac{d\sigma^2}{d-1}\right),
\end{align*}
 {where}
\[
I_d(x;t):=\sum_{k=0}^d (-1)^k \binom{d}{k}\exp\left(\frac{k(d-k)}{2d}t\right)x^{d-k}, \qquad t \ge 0.
\]

\item As $d\rightarrow\infty$, we have 
\[
\mu\llbracket I_d(x;t)\rrbracket \xrightarrow{w} \lambda_t,
\] 
where $\lambda_t$ is the {\it multiplicative free semicircular distribution} (see \cite[Proposition 5]{Biane95} or Section 5 for details).
\end{enumerate}
\end{theorem}

 {As the last main result, we provide alternative proofs for Kabluchko's two limit theorems by using the first combinatorial identities again.}

\begin{theorem}\label{main4}
\begin{enumerate}[\rm (1)]  
\item (Kabluchko \cite{K22}) Let us define $H_d(z;t)$ as the unitary Hermite polynomial. Then we have
\begin{align*}
\mu\llbracket H_d(z;t)\rrbracket \xrightarrow{w} \sigma_t, \qquad d\rightarrow\infty,
\end{align*}
where $\sigma_t$ is the free unitary normal distribution (see \cite{Biane95} and \cite[Lemma 6.3]{BV92}).

\item (Kabluchko \cite{K21}) Let us define $L_{d,m}(z)$ as the unitary Laguerre polynomial. Then we have
\begin{align*}
\mu\llbracket L_{d,m}\rrbracket\xrightarrow{w} \Pi_t, \qquad d\rightarrow\infty,
\end{align*}
where $\Pi_t$ is the free unitary Poisson distribution (see \cite{K21} and \cite[Lemma 6.4]{BV92}).
\end{enumerate}
\end{theorem}

\subsection{Organization of the paper}

 {The remainder of the paper comprises five sections and an appendix. In Section 2, we recall the fundamentals of partitions and combinatorial concepts.
Additionally, the definitions of finite free convolution and finite free cumulant are introduced and discussed for their various known properties.
Furthermore, we provide a formula for the finite free cumulant of finite free multiplicative convolution.}
In Section 3, we prove the combinatorial identities presented in Theorem \ref{main1}. 
In Section 4, we investigate the limit behavior of $D_{1/m}((p_d^{\boxtimes_dm})^{\boxplus_dm})$, as $d,m\rightarrow\infty$ in two cases: (i) $m/d\rightarrow t$ for some $t>0$ and (ii) $m/d\rightarrow0$ (see Theorem \ref{main3}).
In Section 5, we establish the central limit theorem for finite free multiplicative convolution. Additionally, we find its connection to free probability (see Theorem \ref{main5}).
In Section 6, we provide alternative proofs of Kabluchko's two limit theorems (see Theorem \ref{main4}). The key to the proofs in Sections 4--6 lies in the utilization of combinatorial formulas derived from Theorem \ref{main1}. 
In Appendix A, we compute the free cumulants of the multiplicative free semicircular distribution $\lambda_t$, the free unitary normal distribution $\sigma_t$ and the free unitary Poisson distribution $\Pi_t$ via the Lagrange inversion theorem.


\section{Preliminaries}

\subsection{M\"{o}bius functions on partially ordered sets}

In this section, we summarize useful results on the M\"{o}bius functions on partially ordered sets. A {\it partially ordered set} (for short, poset) is a set equipped with a partial order.
More precisely, a pair $P=(P,\le)$ is called a poset if $P$ is a set and $\le $ is a relation on $P$, that is, reflexive, antisymmetric and transitive.
A poset $P$ is said to be finite if the number of elements in $P$ is finite. 


\begin{example}\label{ex:posets}
\begin{enumerate}[\rm (1)]
\item Define $\mathcal{B}(n)$ as the set of all subsets of $[n]$. The set $\mathcal{B}(n)$ can be equipped with the following partial order $\le$: 
\begin{align*}
  A \le B \overset{\text{def}}{\Longleftrightarrow} A\subset B \quad \text{($A$ is a subset of $B$)},
\end{align*}
for $A, B\in \mathcal{B}(n)$.
Therefore $\mathcal{B}(n)=(\mathcal{B}(n),\le )$ is a finite poset. 
It is easy to verify that the minimum and maximum elements of $\mathcal{B}(n)$ are $\emptyset$ and $[n]$, respectively.

\item We call $\pi=\{V_1,\dots, V_r\}$ a {\it partition} of the set $[n]$ if it satisfies that
\begin{enumerate}[\rm (i)]
  \item $V_i$ is a non-empty subset of $[n]$ for all $i=1,2,\dots, r$;
  \item $V_i\cap V_j=\emptyset$ if $i\neq j$;
  \item $V_1\cup \cdots \cup V_r=[n]$.
\end{enumerate}
Each subset $V_i$ is called a {\it block} of $\pi$ and $|V_i|$ denotes the number of elements in $V_i$, namely the {\it size} of $V_i$.

Let $\calP(n)$ be the set of all partitions in $[n]$. The set $\calP(n)$ can be equipped with the following partial order $\le$:
\begin{align*}
  \pi \le \sigma \overset{\text{def}}{\Longleftrightarrow} \text{each block of $\pi$ is completely contained in one of the blocks of $\sigma$}.
\end{align*}
Then $\mathcal{P}(n)=(\mathcal{P}(n),\le )$ is a finite poset. The minimum and maximum elements of $\calP(n)$ are given by $0_n:=\{\{1\},\{2\},\dots, \{n\}\}$ and $1_n:=\{\{1,2,\dots,n\}\}$, respectively.
\end{enumerate}
\end{example}

Let $P=(P,\le )$ be a finite poset. Denote $P^{(2)}:=\{(\pi,\sigma)\in P\times P \ |\ \pi\le \sigma\}$. For $F, G:P^{(2)}\rightarrow \C$, their convolution $F \ast G$ is defined as the function from $P^{(2)}$ to $\C$ by setting
\[
  (F \ast G)(\pi,\sigma):=\sum_{\substack{\rho \in P \\ \pi \le \rho \le \sigma}} F(\pi,\rho) G(\rho,\sigma).
\]
The {\it zeta function} $\zeta_P:P^{(2)}\rightarrow \C$ of $P$ is defined by
\[
  \zeta_P(\pi,\sigma)=1, \qquad (\pi,\sigma)\in P^{(2)}.
\]
The {\it M\"obius function} $\mu_P$ of $P$ is defined as the inverse of $\zeta_P$ with respect to the convolution $\ast$. 

The following inversion principle is one of the most important properties of incidence algebras.

\begin{proposition}[M\"{o}bius inversion formula]\label{prop:Mobius}
Let $P$ be a finite poset. Then there is a unique M\"{o}bius function $\mu_P: P^{(2)} \rightarrow \Z$ such that, for any functions $f,g:P\rightarrow \C$ and $\pi\in P$, the identify
\[
f(\pi)=\sum_{\sigma\le \pi} g(\sigma)
\]
holds, if and only if
\[
g(\pi)=\sum_{\sigma\le \pi} f(\sigma)\mu_P(\sigma,\pi).
\]
\end{proposition}

The following formula on the M\"{o}bius functions is often used in this paper.
\begin{lemma}\label{prop:Mobius on posets}
Let $P$ be a finite poset with the maximum $1_P$ and $\mu_P$ the M\"{o}bius function on $P$. Then the following identity holds:
\begin{align*}
\sum_{\pi \le \sigma} \mu_P(\sigma, 1_P)=\delta_{\pi, 1_P}:=\begin{cases}
1, & \pi=1_P,\\
0, & \pi\neq 1_P.
\end{cases}
\end{align*}
\end{lemma}


\begin{example}\label{ex:Mobius_poset}
\begin{enumerate}[\rm (1)]
\item A function $\mu_{\mathcal{B}(n)}$ denotes the M\"{o}bius function on $\mathcal{B}(n)$. It is easy to verify that, if $W\le V$ in $\mathcal{B}(n)$ then $\mu_{\mathcal{B}(n)}(W, V)=(-1)^{|V|-|W|}$. In particular, we have $\mu_{\mathcal{B}(n)}(W,[n])=(-1)^{n-|W|}$ for all $W\in \mathcal{B}(n)$. 

\item A function $\mu_n=\mu_{\calP(n)}$ denotes the M\"{o}bius function on $\mathcal{P}(n)$. The function $\mu_n$ can be explicitly computed as follows: for $\pi, \sigma \in \calP(n)$, 
\[
  \mu_n(\pi,\sigma)=(-1)^{|\pi|-|\sigma|} (2!)^{r_3} (3!)^{r_4} \cdots ((n-1)!)^{r_n},
\]
where $|\pi|$ denotes the number of blocks of $\pi \in \calP(n)$ and $r_i$ is the number of blocks of $\sigma$ that contain exactly $i$ blocks of $\pi$. In particular, we have
\begin{equation} \label{eq:white}
  \mu_n(\pi,1_n)=(-1)^{|\pi|-1}(|\pi|-1)!
\end{equation}
and
\[
  \mu_n(0_n, \sigma)=(-1)^{n-|\sigma|}(2!)^{t_3}(3!)^{t_4}\cdots ((n-1)!)^{t_n},
\]
where $t_i$ is the number of blocks of $\sigma$ of size $i$.
\end{enumerate}
\end{example}


\subsection{Finite free probability}

 {In this section, we introduce essential concepts in finite free probability. Specifically, we define two finite free convolutions and finite free cumulants, summarizing several results related to these concepts.}


\subsubsection*{ {Finite free convolutions}}

For any $p,q\in \mathbb{P}_{\mathrm{mon}}(d)$, one defines the {\it finite free additive convolution} $p\boxplus_dq$ to be
  \[
    \begin{split}
      (p\boxplus_d q)(x) &= \sum_{k=0}^d (-1)^k \binom{d}{k} \sum_{i+j=k} \frac{k!}{i!j!}\widetilde{a}_i(p) \widetilde{a}_j(q)  x^{d-k}.
    \end{split}
  \]
For $p \in \mathbb{P}_{\mathrm{mon}}(d)$, a polynomial $p^{\boxplus_d m}$ denotes the $m$-th power of finite free additive convolution of $p$. Note that, if $p,q \in \mathbb{P}_{\mathrm{mon}}(d)$ are real-rooted, then so is $p\boxplus_d q\in \mathbb{P}_{\mathrm{mon}}(d)$ (see \cite[Theorem 1.3]{MSS22}).  The finite free additive convolution plays an important role in studying characteristic polynomials of the sum of (random) matrices. For a $d\times d$ real symmetric matrix $A$, $\chi_A$ denotes the characteristic polynomial of $A$. Then we obtain
\[
  (\chi_A\boxplus_d \chi_B)(x)= \mathbb{E}_Q \det[xI_d-A-QBQ^*],
\]
where the expectation is taken over unitary matrices $Q$ distributed uniformly on the unitary group of degree $d$ (see \cite[Theorem 1.2]{MSS22}).
Moreover, the finite additive convolution is closely related to the free additive convolution $\boxplus$ which describes the law of sum of freely independent non-commutative random variables (see \cite{BV93, M92, V86} for detailed information on free additive convolution).
\begin{proposition} (see \cite[Corollary 5.5]{AP})
Suppose that $p_d,q_d \in \mathbb{P}_{\mathrm{mon}}(d)$ are real-rooted and $\mu,\nu$ are probability measures on $\R$ with compact support. If $\mu\llbracket p_d \rrbracket\xrightarrow{w}\mu$ and $\mu\llbracket q_d \rrbracket \xrightarrow{w} \nu$ as $d\rightarrow\infty$, respectively, then $\mu \llbracket p_{d}\boxplus_d q_{d} \rrbracket \xrightarrow{w}\mu\boxplus\nu$ as $d\rightarrow\infty$.
\end{proposition}

Similarly, for $p,q \in \mathbb{P}_{\mathrm{mon}}(d)$, the {\it finite free multiplicative convolution} $p\boxtimes_dq$ is defined as
\begin{align*}
  (p\boxtimes_d q)(x)= \sum_{i=0}^d (-1)^i \binom{d}{i} \widetilde{a}_i(p) \widetilde{a}_i(q) x^{d-i}.
\end{align*}
For $p \in \monP(d)$, we denote by $p^{\boxtimes_d m}$ the $m$-th power of finite free multiplicative convolution of $p$.
Note that, if $p$ has only nonnegative roots and $q$ is real-rooted, then $p\boxtimes_d q$ has only real roots (see, e.g., \cite[Section 16, Exercise 2]{Mar66}).
If $p,q$ have only roots located on $\mathbb{T}:=\{z\in \C \ |\ |z|=1\}$, then so is $p\boxtimes_d q$ (see, e.g., \cite[Satz 3]{Sz22}).
According to \cite[Theorem 1.5]{MSS22}, the finite free multiplicative convolution describes a characteristic polynomial of the product of positive definite matrices. More precisely, if $A$ and $B$ are $d\times d$ positive definite matrices, then
\[
  (\chi_A\boxtimes_d \chi_B)(x)=\mathbb{E}_Q\det[xI_d-AQBQ^*].
\]
Furthermore, it is also known that $\boxtimes_d$ is closely related to free multiplicative convolution $\boxtimes$ which describes the law of multiplication of freely independent random variables (see \cite{BV93} for definition and known facts of free multiplicative convolution).
\begin{proposition} (see \cite[Theorem 1.4]{AGVP})
Let us consider $p_d, q_d \in \monP(d)$ in which $p_d$ has only nonnegative roots and $q_d$ is real-rooted. Further, consider probability measures $\mu,\nu$ on $\R$ with compact support, in which $\mu$ is supported on $[0,\infty)$. If $\mu \llbracket p_{d} \rrbracket \xrightarrow{w} \mu$ and $\mu \llbracket q_{d} \rrbracket \xrightarrow{w} \nu$ as $d\rightarrow\infty$, respectively, then $\mu \llbracket p_{d}\boxtimes_d q_{d} \rrbracket \xrightarrow{w} \mu \boxtimes \nu$ as $d\rightarrow\infty$. 
\end{proposition}
Also, according to \cite[Proposition 2.9]{K21}, the same statement holds when $p_d,q_d$ have only roots located on $\mathbb{T}$ and $\mu,\nu$ are probability measures on $\mathbb{T}$.

\subsubsection*{ {Limit behavior of finite free multiplicative convolution power}}

 {In Section 4, we will investigate limit theorems regarding finite free multiplicative convolutions.  Howeverm before doing so, here, we conssider the limit behavior of the $m$-fold finite free multiplicative convolution $p^{\boxtimes_dm}$ of $p \in \monP(d)$ with nonnegative roots, which surprsingly may lead to a non-trivial limit. }

First, let us recall Newton's inequality and Maclaurin's inequality for $\{\widetilde{a}_i(p)\}_{i=0}^d$. (see, e.g., \cite[Section 2.22]{HLP34}).
\begin{lemma}\label{lem:NewMac}

\begin{enumerate}[\rm (1)]
\item (Newton's inequality) \label{prop:Newton}
  Let $p \in \monP(d)$ be a monic polynomial with real roots.
  Then
   \[
    \widetilde{a}_{i+1}(p)\widetilde{a}_{i-1}(p) \le \widetilde{a}_{i}(p)^2, \qquad i=1,2,\dots, d-1.
  \]
  The equality holds if and only if its roots are the same $\alpha$ in which case $\widetilde{a}_{i}(p) = \alpha^i$ for $i=1,2,\dots, d$.
  
  \item (Maclaurin's inequality) Let $p \in \monP(d)$ be a monic polynomial with positive roots.
  Then 
  \begin{equation} \label{eq:Maclaurin}
    \widetilde{a}_{1}(p) \ge \widetilde{a}_{2}(p)^{\frac{1}{2}} \ge \cdots \ge \widetilde{a}_{d}(p)^{\frac{1}{d}}
  \end{equation}
  with equality if and only if its roots are the same $\alpha$.
  \end{enumerate}
\end{lemma}

\begin{remark}
   {The inequality \eqref{eq:Maclaurin} remains valid even when $p$ has zero roots.}
  More precisely, if $p$ has exactly $k$ zero roots then $\widetilde{a}_{1}(p) >\widetilde{a}_{2}(p)^{\frac{1}{2}} > \cdots > \widetilde{a}_{d-k}(p)^{\frac{1}{d-k}}$ and $\widetilde{a}_{d-k+1}(p) = \cdots = \widetilde{a}_{d}(p) = 0$.
\end{remark}

\begin{proposition}\label{thm:newlimit}
  For $p\in\mathbb{P}_{\mathrm{mon}}(d)$ with nonnegative roots, we have
  \begin{align*}
    \lim_{m\rightarrow\infty}p^{\boxtimes_dm} (x) =
    \begin{cases}
      x^d,          & \widetilde{a}_1(p)<1,                                   \\
      x^d-dx^{d-1}, & \widetilde{a}_1(p)=1 \text{ and } \widetilde{a}_2(p)<1,  \\
      (x-1)^d,      & \widetilde{a}_1(p)=1 \text{ and } \widetilde{a}_2(p)=1.
    \end{cases}
  \end{align*}
  The limit does not exist when $\widetilde{a}_1(p)>1$.
\end{proposition}
\begin{proof}
  For each $i$, we have
  \begin{align*}
    \widetilde{a}_i (p^{\boxtimes_dm}) =\widetilde{a}_i(p)^m \xrightarrow{m\rightarrow\infty}
    \begin{cases}
      0,            & \widetilde{a}_i(p)<1,  \\
      1, & \widetilde{a}_i(p)=1,  \\
      \infty,       & \widetilde{a}_i(p)>1.
    \end{cases}
  \end{align*}

  Since $d$ is fixed, it does not lead to any limit when $\widetilde{a}_1(p)>1$.

  Next, we consider the case that $\widetilde{a}_1(p) \le 1$.
  \begin{itemize}
    \item If $\widetilde{a}_1(p) < 1$, then $\widetilde{a}_i(p) < 1$ for $i=2,\dots, d$ by Lemma \ref{lem:NewMac} (2). 
    Then we have
          \begin{align*}
            p^{\boxtimes_dm} (x)=\sum_{i=0}^d (-1)^i \binom{d}{i} \widetilde{a}_i(p)^m x^{d-i} \xrightarrow{m\rightarrow\infty} x^d.
          \end{align*}

    \item If $\widetilde{a}_1(p)=1$, then $\widetilde{a}_2(p)\le1$ by Lemma \ref{lem:NewMac} (2). 
    There are two possible cases.
          \begin{enumerate}[\rm (i)]
            \item If $\widetilde{a}_2(p)<1$, then $\widetilde{a}_i(p)<1$ for $i=2,\dots, d$, and therefore
                  \[
                    p^{\boxtimes_dm} (x)=\sum_{i=0}^d (-1)^i \binom{d}{i} \widetilde{a}_i(p)^m x^{d-i} \xrightarrow{m\rightarrow\infty} x^d-dx^{d-1}.
                  \]

            \item If $\widetilde{a}_2(p)=1$, then $\widetilde{a}_i(p)=1$ for $i=2,\dots, d$ by Lemma \ref{lem:NewMac} (1).
            It means $p(x)=(x-1)^d$.
            Thus we have
                  \[
                    p^{\boxtimes_dm} (x)=(x-1)^d\xrightarrow{m\rightarrow\infty} (x-1)^d.
                  \]
          \end{enumerate}
  \end{itemize}
\end{proof}

\subsubsection*{ {Finite free cumulants}}

There is a useful combinatorial concept for understanding two finite free convolutions. For $p\in \mathbb{P}_{\mathrm{mon}}(d)$, the {\it finite free cumulant} of $p$ is defined by
\begin{align}\label{def:cumulant}
  \kappa_n^{(d)}(p):=\frac{(-d)^{n-1}}{(n-1)!} \sum_{\pi \in \calP(n)} \widetilde{a}_\pi(p) \mu_n(\pi,1_n),
\end{align}
for $n=1,2,\dots, d$ (see \cite[Proposition 3.4]{AP} for details).

\begin{example}
  \begin{enumerate}[\rm (1)]
    \item
          Let us take $p(x)=x^d-dx^{d-1}$. Since $\widetilde{a_1}(p) =1$ and $\widetilde{a}_i(p)=0$ for all $i=2,\dots, d$, it is easy to see that $\kappa_n^{(d)}(p)=d^{n-1}$ for $n=1,2,\dots, d$.

    \item Let $\lambda > 0$. We define the normalized Laguerre polynomial $\widehat{L}_d^{(\lambda)}$ as
          \[
            \widehat{L}_d^{(\lambda)}(x)=\sum_{i=0}^d (-1)^i \binom{d}{i} \frac{(d\lambda)_i}{d^i} x^{d-i},
          \]
          where $(\alpha)_i:=\alpha(\alpha-1)\cdots (\alpha-i+1)$. Then the finite free cumulants of $\hat{L}_d^{(\lambda)}$ are given by $\kappa_n^{(d)}(\widehat{L}_d^{(\lambda)})=\lambda$ for $n=1,2, \dots, d$ (see \cite{AGVP}).
  \end{enumerate}
\end{example}

According to \cite[Proposition 3.6]{AP}, it is known that the finite free cumulant linearizes the finite free additive convolution:
\begin{align*}
  \kappa_n^{(d)}(p\boxplus_d q ) =\kappa_n^{(d)}(p)+ \kappa_n^{(d)}(q)
\end{align*}
for $p, q\in \monP(d)$. In particular, we have
\begin{align*}
  \kappa_n^{(d)}(p^{\boxplus_d m})=m \kappa_n^{(d)}(p), \qquad m\in \N.
\end{align*}

Recall that, for $p \in \monP(d)$ and $c\neq 0$,
\begin{align*}
  D_c (p)(x)= c^d p(x/c).
\end{align*}
Hence, by the definition of finite free cumulants, we obtain
\begin{align}\label{eq:kappa_D}
  \kappa_n^{(d)}(D_c(p))=c^n \kappa_n^{(d)}(p).
\end{align}

According to \cite{AP}, it is known that the finite free cumulants approach the free cumulants introduced by Speicher as the degree $d$ tends to infinity.
A consequence of this is the following criterion for convergence in distribution.
\begin{proposition}\label{prop:converges_freecumulant}
Let us consider $p_d\in \mathbb{P}_{\mathrm{mon}}(d)$ and a probability measure $\mu$ with compact support. The following assertions are equivalent.
\begin{enumerate}[\rm (1)]
\item $\mu \llbracket p_d \rrbracket \xrightarrow{w}\mu $ as $d\rightarrow\infty$.
\item For all $n\in \N$, $\lim_{d\rightarrow \infty} \kappa_{n}^{(d)}(p_d)=\kappa_{n}(\mu)$.
\end{enumerate}
\end{proposition}
\begin{proof}
 {We give a sketch of the proof.
First we assume that $\mu\llbracket p_d \rrbracket \xrightarrow{w} \mu$ as $d\to \infty$.
Note that this weak convergence is equivalent to $\lim_{d\to \infty}m_n(\mu\llbracket p_d \rrbracket ) = m_n(\mu)$ for any $n\in \N$, where $m_n(\rho)$ denotes the $n$-th moment of a probability measure $\rho$ on $\R$. Then Theorem 5.4 in \cite{AP} implies the statement (2). Conversely, we assume that $\lim_{d\rightarrow \infty} \kappa_{n}^{(d)}(p_d)=\kappa_{n}(\mu)$ for any $n\in\N$. 
According to the equality in the line following (5.2) in [2, Theorem 5.4], we obtain
\[
\lim_{d\to \infty} m_n(\mu\llbracket p_d \rrbracket )  = \sum_{\sigma \in NC(n)} \kappa_\sigma (\mu) = m_n(\mu),
\]
where $NC(n)$ denotes the set of non-crossing partitions of $[n]$ and the last equality holds due to the moment-cumulant formula (see \cite[Proposition 11.4]{NS}). This implies that $\mu\llbracket p_d \rrbracket \xrightarrow{w} \mu$ as $d\to \infty$.}
\end{proof}

\subsubsection*{ {Formulas for finite free cumulants related to the finite free multiplicative convolution}}

In the following, we provide a formula for finite free cumulant of $p^{\boxtimes_d m}$ for $p \in \monP(d)$ and $m\in \N$.
First, the following lemma is directly derived from the definition of finite free multiplicative convolution.
\begin{lemma}\label{lem:finite_free_multi} (see \cite[Lemma 3.1]{FU})
  For a family $\{p_i\}_{i=1}^m \subset \mathbb{P}_{\mathrm{mon}}(d)$, one has
  \begin{align*}
    \widetilde{a}_j(p_1 \boxtimes_d \cdots \boxtimes_d p_m)=\prod_{i=1}^m\widetilde{a}_j(p_i) \qquad (j=1, \dots, d).
  \end{align*}
  In particular, if all $p_i = p$ for some $p \in \monP(d)$, then
  \begin{align*}
    \widetilde{a}_j(p^{\boxtimes_dm})=\widetilde{a}_j(p)^m.
  \end{align*}
\end{lemma}
Second, it is also known that
\[
  \kappa_n^{(d)}(p\boxtimes_d q) = \frac{(-d)^{n-1}}{(n-1)!} \sum_{\substack{\sigma, \tau \in \calP(n) \\ \sigma \lor \tau=1_n}} d^{|\sigma|+|\tau|-2n} \mu_n(0_n,\sigma) \mu_n(0_n,\tau) \kappa_{\sigma}^{(d)}(p) \kappa_{\tau}^{(d)}(q)
\]
from \cite[Theorem 1.1]{AGVP}.
In particular, if $p=q$, then
\[
  \kappa_n^{(d)}(p^{\boxtimes_d2})=\frac{(-d)^{n-1}}{(n-1)!} \sum_{\substack{\sigma_1, \sigma_2 \in \calP(n) \\ \sigma_1 \lor \sigma_2=1_n}} \left( \prod_{i=1}^2 d^{|\sigma_i|-n} \mu_n(0_n,\sigma_i) \kappa_{\sigma_i}^{(d)}(p)\right).
\]
In general, the following formula holds.
\begin{proposition}\label{prop:cumulant_formula}
  For a family $\{p_i\}_{i=1}^m \subset \monP(d)$, we have
  \begin{align*}
    \kappa_n^{(d)}(p_1 \boxtimes_d \cdots \boxtimes_d p_m)
    & =\frac{(-d)^{n-1}}{(n-1)!} \sum_{\pi \in \calP(n)} \prod_{i=1}^m \sum_{\sigma_i\le \pi} d^{|\sigma_i|-n} \mu_n(0_n,\sigma_i) \kappa_{\sigma_i}^{(d)}(p_i) \mu_n(\pi,1_n) \\
    & =\frac{(-d)^{n-1}}{(n-1)!}\sum_{\substack{ \sigma_1,\dots, \sigma_m \in \calP(n) \\ \sigma_1\lor \cdots\lor \sigma_m=1_n}} \left( \prod_{i=1}^m d^{|\sigma_i|-n} \mu_n(0_n,\sigma_i) \kappa_{\sigma_i}^{(d)}(p_i)\right).
  \end{align*}
  In particular, if all $p_i = p$ for some $p \in \monP(d)$, then
  \begin{align*}
    \kappa_n^{(d)}(p^{\boxtimes_dm})
    & =\frac{(-d)^{n-1}}{(n-1)!} \sum_{\pi \in \calP(n)} \left(\sum_{\sigma\le \pi} d^{|\sigma|-n} \mu_n(0_n,\sigma) \kappa_\sigma^{(d)}(p)\right)^m \mu_n(\pi,1_n) \\
    & =\frac{(-d)^{n-1}}{(n-1)!}\sum_{\substack{ \sigma_1,\dots, \sigma_m \in \calP(n) \\ \sigma_1\lor \cdots\lor \sigma_m=1_n}} \left( \prod_{i=1}^m d^{|\sigma_i|-n} \mu_n(0_n,\sigma_i) \kappa_{\sigma_i}^{(d)}(p)\right).
  \end{align*}
\end{proposition}

\begin{proof}
By the definition of finite free cumulant \eqref{def:cumulant} and Lemma \ref{lem:finite_free_multi}, we obtain
\begin{align}
  \kappa_n^{(d)}(p_1 \boxtimes_d \cdots \boxtimes_d p_m)
  &=\frac{(-d)^{n-1}}{(n-1)!} \sum_{\pi \in \calP(n)} \widetilde{a}_\pi(p_1 \boxtimes_d \cdots \boxtimes_d p_m) \mu_n(\pi,1_n) \notag \\
  &=\frac{(-d)^{n-1}}{(n-1)!} \sum_{\pi \in \calP(n)} \prod_{i=1}^m \widetilde{a}_\pi(p_i) \mu_n(\pi,1_n).\label{eq:cumulant_formula}
\end{align}
  According to \cite[Proposition 3.4]{AP}, since
  \[
    \widetilde{a}_\pi(p)=\sum_{\sigma\le \pi} d^{|\sigma|-n} \mu_n(0_n,\sigma) \kappa_\sigma^{(d)}(p),
  \]
  for any $p \in \monP(d)$,
  the first equality holds by \eqref{eq:cumulant_formula}.

  The second equality is proved as follows:
  \begin{align*}
    \sum_{\pi \in \calP(n)} & \prod_{i=1}^m \sum_{\sigma_i\le \pi} d^{|\sigma_i|-n} \mu_n(0_n,\sigma_i) \kappa_{\sigma_i}^{(d)}(p_i) \mu_n(\pi,1_n)\\
    & =\sum_{\pi \in \calP(n)}\left(\sum_{\sigma_1,\dots, \sigma_m \le \pi} \prod_{i=1}^md^{|\sigma_i|-n} \mu_n(0_n,\sigma_i) \kappa_{\sigma_i}^{(d)}(p_i)\right)\mu_n(\pi,1_n)\\
    & =\sum_{\sigma_1,\dots, \sigma_m\in \calP(n)} \prod_{i=1}^m d^{|\sigma_i|-n} \mu_n(0_n,\sigma_i) \kappa_{\sigma_i}^{(d)}(p_i) \left( \sum_{\sigma_1\lor \cdots\lor\sigma_m \le \pi} \mu_n(\pi,1_n)\right) \\
    & =\sum_{\substack{ \sigma_1,\dots, \sigma_m \in \calP(n)\\ \sigma_1\lor \cdots\lor \sigma_m=1_n}} \left( \prod_{i=1}^m d^{|\sigma_i|-n} \mu_n(0_n,\sigma_i) \kappa_{\sigma_i}^{(d)}(p_i)\right),
  \end{align*}
  where we used Proposition \ref{prop:Mobius on posets} on the third line.
\end{proof}

\begin{example}
  By using the first equation in Proposition \ref{prop:cumulant_formula}, we obtain
  \[
    \kappa_n^{(d)}\left((\widehat{L}_d^{(1)})^{\boxtimes_dm} \right) =\frac{(-d)^{n-1}}{(n-1)!} \sum_{\pi \in \calP(n)} \left(\sum_{\sigma\le \pi} d^{|\sigma|-n} \mu_n(0_n,\sigma) \right)^m \mu_n(\pi,1_n)
  \]
  because $\kappa_n^{(d)}(\widehat{L}_d^{(1)})=1$ for all $n$. The formula implies that
  \begin{align*}
    \kappa_2^{(d)}\left( (\widehat{L}_d^{(1)})^{\boxtimes_dm}\right)=d\left\{1-\left(1-\frac{1}{d}\right)^m\right\}
  \end{align*}
  and
  \begin{align*}
    \kappa_3^{(d)}\left((\widehat{L}_d^{(1)})^{\boxtimes_dm}\right)=d^2 \left\{1-\frac{3}{2}\left(1-\frac{1}{d}\right)^m+\frac{1}{2}\left(1-\frac{1}{d}\right)^m\left(1-\frac{2}{d}\right)^m \right\}.
  \end{align*}
\end{example}


\section{Combinatorial formula related to finite free probability}

In this section, we investigate the value of
\begin{align*}
\sum_{\pi \in \mathcal{P}(n)} \prod_{i=1}^k \left( \sum_{V\in \pi} f_i(|V|) \right)\mu_n(\pi,1_n)
\end{align*}
for polynomials $f_1, \dots, f_k$ without constant terms.
 {These values will be crucial in examining the convergence of finite free cumulants in Sections 4--6.}

\subsection{The definition of polynomials $r_k[f](z)$ and $s_k[f](z)$}

\begin{definition}
For a polynomial $f(x) \in \C[x]_0$, define the family $\{\varphi_n[f](t)\}_{n\in\N}$ and its generating series $M[f](t,z)$ as follows. 
For $n\in\N$
\[
  \varphi_n[f](t) := \exp(f(n)t),
\]
and
\[
M[f](t,z):= 1 + \sum_{n=1}^\infty \frac{\varphi_n[f](t)}{n!}z^n.
\]
Also, we define $\{\psi_n[f](t)\}_{n\in \N}$, $\{r_k[f](z)\}_{k\in\N}$ and $\{s_k[f](z)\}_{k\in\N}$ as characterizing the following identities:
\[
  M[f](t,z) = e^z\left(1 + \sum_{k=1}^\infty \frac{r_k[f](z)}{k!}t^k\right),
\]
and
\[
  \log(M[f](t,z)) = \sum_{n=1}^\infty \frac{\psi_n[f](t)}{n!}z^n = z + \sum_{k=1}^\infty \frac{s_k[f](z)}{k!}t^k.
\]
\end{definition}
\begin{lemma}
  Let $\{\varphi_n[f](t)\}_{n\in \N}$, $\{\psi_n[f](t)\}_{n\in\N}$, $\{r_k[f](z)\}_{k\in \N}$, and $\{s_k[f](z)\}_{k\in\N}$ be defined as above.
  
  \begin{enumerate}[\rm (1)]
      \item For all $n,k\in\N$, we obtain
    \[
    \varphi_n[f](t) = \sum_{\pi \in \calP(n)} \psi_\pi[f](t), \qquad \text{or equivalently \qquad} \psi_n[f](t) = \sum_{\pi \in \calP(n)}\varphi_\pi[f](t) \mu_n(\pi,1_n),
    \]
    and
    \begin{equation} \label{eq:moment-cumulant_rs}
      r_k[f](z) = \sum_{\pi \in \calP(k)} s_\pi[f](z), \qquad \text{or equivalently \qquad} s_k[f](z) = \sum_{\pi \in \calP(k)}r_\pi[f](z) \mu_k(\pi,1_k).
          \end{equation}
    \item For all $n,k\in\N$, we have
    \begin{align}\label{sk_desired}
    s_k^{(n)}[f](0) = \psi_n^{(k)}[f](0) = \sum_{\pi \in \calP(n)}\left(\sum_{V\in\pi}f(|V|)\right)^k \mu_n(\pi,1_n),
    \end{align}
    where $s_k^{(n)}[f](0):= \frac{d^n}{dz^n} s_k[f](z) |_{z=0}$ and $\psi_n^{(k)}[f](0):=\frac{d^k}{dt^k}\psi_n[f](t)|_{t=0}$.
    \item For all $n,k\in\N$,
    \begin{equation} \label{eq:def_of_r}
      r_k^{(n)}[f](0) = \sum_{l=1}^n \binom{n}{l}(-1)^{n-l} f(l)^k,
    \end{equation}
    where $r_k^{(n)}[f](0):=\frac{d^n}{dz^n}r_k[f](z)|_{z=0}$.
  \end{enumerate}
\end{lemma}

\begin{proof}
The statement (1) follows from the moment-cumulant formula. Since $\varphi_{\pi}[f](0)=1$ and
\begin{align*}
\varphi_{\pi}^{(k)}[f](t):=\frac{d^k}{dt^k}\varphi_\pi[f](t)=\left(\sum_{V\in \pi} f(|V|)\right)^k \varphi_{\pi}[f] (t)
\end{align*}
for $k\in \N$ and $\pi\in \calP(n)$, we obtain
\begin{align*}
\psi_n^{(k)}[f](0)= \sum_{\pi \in \calP(n)}\varphi_\pi^{(k)}[f](0) \mu_n(\pi,1_n) = \sum_{\pi\in \calP(n)} \left(\sum_{V\in \pi} f(|V|)\right)^k \mu_n(\pi,1_n).
\end{align*}
Because we just exchanged the order of sum, we have $s_k^{(n)}[f](0) = \psi_n^{(k)}[f](0)$ as a desired result in (2). By the definitions of $M[f](t,z)$ and $r_k[f](z)$,
\begin{align*}
r_k^{(n)}[f](0) &= \sum_{l=1}^n \binom{n}{l}(-1)^{n-l} \varphi_l^{(k)}[f](0).
\end{align*}
It is easy to verify that $\varphi_l^{(k)}[f](0)=f(l)^k$ as desired.
\end{proof}

\begin{example} \label{example:mosteasy}
  Here, as the simplest polynomial, take $f(x) = x$.
  Then it immediately follows that $\varphi_n[f](t) = \exp(nt)$, $\psi_1[f](t) = e^t$ and $\psi_n[f](t) = 0$ $(n\ge 2)$.
  Hence, $s_k[f](z) = z$ and
  \begin{equation} \label{eq:sum_pi}
    r_k[f](z) = \sum_{\pi\in\calP(k)} z^{|\pi|}
  \end{equation}
  for every $k\in\N$.
\end{example}

\begin{remark} \label{rem:kumo}
  An easy consequence of \eqref{eq:def_of_r} and \eqref{eq:sum_pi} is
  \begin{equation} \label{eq:kurage}
    \sum_{l=1}^n \binom{n}{l}(-1)^{n-l} l^k= 0
  \end{equation}
  for $k\in\N$, if $n>k$.
  Hence, for any polynomial $f \in \C[x]_0$, we can observe that $r_k[f](z)$ is a polynomial because $r_k^{(n)}[f](0) = 0$ if $n > k \deg(f)$ from \eqref{eq:kurage}.
  Also, $s_k[f](z)$ is a polynomial due to the moment--cumulant formula \eqref{eq:moment-cumulant_rs}.
\end{remark}


\begin{lemma}\label{lem:linear}
Let $f,g$ be polynomials in $\C[x]_0$,\ $\alpha,\beta\in \C$ and $k\in \N$. Then
\begin{enumerate}[\rm (1)]
\item $r_1[\alpha f+ \beta g](z)=\alpha r_1[f] (z)+ \beta r_1[g](z)$.
\item $r_k[f](z)=r_1[f^k](z)$.
\end{enumerate}
\end{lemma}
\begin{proof}
  Both are derived directly from \eqref{eq:def_of_r}.
\end{proof}

Next, we will generalize the definition of polynomials $r_k[f](z)$ and $s_k[f](z)$ as determined by polynomials $\{f_i\}_{i=1}^k$ and satisfying multi-linearity.
According to Lemma \ref{lem:linear}, the following can be understood as a natural extension.

\begin{definition}
  Let us consider $f_1,\dots, f_k\in \C[x]_0$. We define
\begin{equation} \label{eq:def_r}
  r_k[f_1,\dots,f_k] (z):= r_1[f_1 \cdots f_k] (z).
\end{equation}
  Likewise, for $\pi\in\calP(k)$, 
  \[
    r_\pi[f_1,\dots,f_k](z) := \prod_{V = \{i_1,\dots,i_{|V|}\}\in\pi} r_{|V|}[f_{i_1},\dots,f_{i_{|V|}}](z).
  \]
  We additionally define
  \begin{equation} \label{eq:cumulant_moment_s}
    s_\pi[f_1,\dots,f_k](z) := \sum_{\substack{\sigma\in\calP(k)\\\sigma \le \pi}} r_\sigma[f_1,\dots,f_k](z) \mu_k(\sigma,\pi)
  \end{equation}
  for $\pi \in \calP(k)$. In particular, $s_k[f_1,\dots,f_k](z)$ denotes $s_{1_k}[f_1,\dots,f_k](z)$.
\end{definition}

Clearly, $s_k[f, \dots, f] = s_k[f]$ for all polynomial $f\in \C[x]_0$, and hence this is a generalization.
The benefits of this generalization are the subsequent properties.

\begin{lemma}
  For $k\in \N$ and $f_1,\dots, f_k \in \C[x]_0$, we have
  \[
    r_k[f_1,\dots,f_k] (z)= \sum_{\pi\in\calP(k)} s_\pi[f_1,\dots,f_k] (z).
  \]
  Likewise, we obtain
  \[
    s_\pi[f_1,\dots,f_k](z) = \prod_{V = \{i_1,\dots,i_{|V|}\}\in\pi} s_{|V|}[f_{i_1},\dots,f_{i_{|V|}}](z)
  \]
  for $\pi\in\calP(k)$.
\end{lemma}
\begin{proof}
It follows from the standard discussion using the M\"{o}bius inversion formula, see \cite[Lectures 9--11]{NS}.
\end{proof}

\begin{example} \label{ex:coffee}
  As an interesting example, we take the polynomials $g_m(x) = x^m$ for $m\in\N$.
  Let us consider $s_k[g_{m_1}, \dots, g_{m_k}](z)$ for positive integers $\{m_i\}_{i=1}^k$ and $M = \sum_{i=1}^k m_i$.
  First, note that $g_1(x) = x$ and also 
  \[
    \begin{split}
      r_k[g_{m_1}, \dots, g_{m_k}] (z)&= r_{1}[g_{M}] (z)\\
      &= r_{M}[g_{1}] (z)\\
      &= \sum_{\pi\in\calP(M)} z^{|\pi|}
    \end{split}
  \]
  due to Lemma \ref{lem:linear} and Equation \eqref{eq:sum_pi}.
  
  The map $\calP(k) \to \calP(M)$, $\sigma \mapsto \widehat{\sigma}$, is defined as each point $\{i\}$ expanding $m_i$-interval, e.g., $\widehat{0}_k = \{\{1, \dots, m_1\}, \{m_1+1, \dots, m_1 + m_2\}, \dots, \{M - m_k+1, \dots, M\}\}$; in particular, $\calP(k)$ and $[\widehat{0}_k, 1_{M}]$ are poset isomorphism via this map. Then we have
  \[
    \begin{split}
      s_k[g_{m_1}, \dots, g_{m_k}] (z)&= \sum_{\sigma\in\calP(k)} r_\sigma[g_{m_1}, \dots, g_{m_k}](z)\mu_k(\sigma,1_k)\\
      &= \sum_{\sigma\in\calP(k)} r_{\widehat{\sigma}}[g_1,\dots,g_1] (z)\mu_{M}(\widehat{\sigma},1_{M})\\
      &=\sum_{\sigma\in\calP(k)} \sum_{\substack{\pi \in \calP(M) \\ \pi \le \widehat{\sigma}}} z^{|\pi|} \mu_{M}(\widehat{\sigma},1_{M})\\
      &=\sum_{\pi \in \calP(M)} z^{|\pi|} \sum_{\substack{\sigma \in \calP(k) \\ \pi \le \widehat{\sigma}}} \mu_{M}(\widehat{\sigma},1_{M})\\
      &=\sum_{\pi \in \calP(M)} z^{|\pi|} \sum_{\substack{\rho\in \calP(k)\\ \pi \vee \widehat{0}_k  \le \rho}} \mu_{M}(\rho,1_{M})\\
      &= \sum_{\substack{\pi \in \calP(M) \\ \pi \vee \widehat{0}_k = 1_{M}}} z^{|\pi|},
    \end{split}
  \]
  where we used the result of Example \ref{example:mosteasy} on the second line and Lemma \ref{prop:Mobius on posets} on the fifth line.
  Thus, $\deg (s_k[g_{m_1}, \dots, g_{m_k}]) = M - (k-1)$ and $\lead (s_k[g_{m_1}, \dots, g_{m_k}] )= \# \{\pi \in \calP(M) \ |\ \pi \vee \widehat{0}_k = 1_{M}, |\pi| = M - (k-1)\}$.
\end{example}

We obtain a generalized result of \eqref{sk_desired} as follows. 

\begin{proposition} \label{prop:multiple}
For any $k,n\in \N$ and $f_1,\dots, f_k \in \C[x]_0$,
  \begin{align*}
    s_k^{(n)}[f_1,\dots,f_k](0) : &= \frac{d^n}{dz^n} s_k[f_1,\dots,f_k](z) |_{z=0}\\
    &=\sum_{\pi\in\calP(n)} \prod_{i=1}^k\left(\sum_{V\in\pi}f_i(|V|)\right) \mu_n(\pi,1_n).
  \end{align*}
\end{proposition}


 {For the proof of Proposition \ref{prop:multiple}, we need to introduce the concept of symmetric multilinear maps on vector spaces and relevant results.}
Let $V, W$ be vector spaces over $\C$.
A multi-linear map $\Phi: V^k\rightarrow W$ is said to be {\it symmetric} if 
\[
\Phi(x_1,\dots, x_k)=\Phi(x_{\iota(1)},\dots, x_{\iota(k)}), \qquad (x_1,\dots, x_k)\in V^k,
\]
for any permutation $\iota$ of $[k]$.

\begin{lemma}[see \cite{Tho13}] \label{lem:core}
Let $V, W$ be vector spaces over $\C$.
If $\Phi: V^k\rightarrow W$ is a symmetric multi-linear map, then it is written as
\begin{align*}
\Phi(x_1,\dots, x_k)=\frac{1}{k!} \sum_{l=1}^k (-1)^{k-l} \sum_{\substack{J \subset [k]\\ |J| = l}} \Phi(\underbrace{x_J,\dots, x_J}_{k \text{ times}}),
\end{align*}
where $x_J:=\sum_{j\in J} x_j$.
\end{lemma}

\begin{lemma} \label{prop:multi-linear}
 The both maps
 \[
 (f_1,\dots, f_k)\mapsto r_k[f_1,\dots, f_k](z) \quad \text{and} \quad
 (f_1,\dots, f_k) \mapsto s_k[f_1,\dots, f_k](z)
 \] 
 are symmetric multi-linear maps from $\C[x]_0^k$ to $\C[z]$.
\end{lemma}
\begin{proof}
  By the definition \eqref{eq:def_r} and Lemma \ref{lem:linear}, it is clear that the first map is symmetric and multi-linear.

 The multi-linearity of the second map follows from the moment--cumulant formula \eqref{eq:cumulant_moment_s}.
 For a permutation $\iota$ of $[k]$,
 \begin{align*}
 s_k[f_{\iota(1)},\dots, f_{\iota(k)}](z) &= \sum_{\sigma \in \calP(k)} r_\sigma[f_{\iota(1)},\dots, f_{\iota(k)}](z) \mu_k(\sigma,1_k)\\
 &=\sum_{\sigma \in \calP(k)} r_{\iota^{-1}(\sigma)}[f_1,\dots, f_k](z) \mu_k(\sigma,1_k),
 \end{align*}
 where $\iota^{-1}(\sigma) := \{\iota^{-1}(V)\mid V \in \sigma \}$.
 Using the change of variables $\pi = \iota^{-1}(\sigma)$, we have
 \begin{align*}
  s_k[f_{\iota(1)},\dots, f_{\iota(k)}](z) &=\sum_{\pi \in \calP(k)} r_{\pi}[f_1,\dots, f_k](z) \mu_k(\iota(\pi),1_k) \\
  &=\sum_{\pi \in \calP(k)} r_{\pi}[f_1,\dots, f_k](z) \mu_k(\pi,1_k) \\
  &=s_k[f_1,\dots, f_k](z),
\end{align*}
where we used the equality $\mu_k(\iota(\pi),1_k) = \mu_k(\pi,1_k)$ due to Equation \eqref{eq:white}.
Thus, the second map is symmetric.
\end{proof}


\begin{proof}{\bf (Proposition \ref{prop:multiple})} \
Let us fix $k, n\in \N$. Define 
\[
\Phi_1(f_1,\dots, f_k):=s_k^{(n)}[f_1,\dots,f_k](0)
\]
 for $(f_1,\dots, f_k) \in \C[x]_0^k$.
 Then the map $\Phi_1$ is a symmetric multi-linear map from $\C[x]_0^k$ to $\C$ by Lemma \ref{prop:multi-linear}.
  Also, we define a symmetric multi-linear map $\Phi_2: \C[x]^{k}_0\rightarrow \C$ as
  \[
    \Phi_2(f_1,\dots, f_k) := \sum_{\pi\in\calP(n)} \prod_{i=1}^k\left(\sum_{V\in\pi}f_i(|V|)\right) \mu_n(\pi,1_n).
  \]
  One can see that, for any $f\in\C[x]_0$, we have $\Phi_1(f,\dots, f) = \Phi_2(f,\dots,f)$ due to Equation \eqref{sk_desired}.
  Thus, Lemma \ref{lem:core} implies that for any $(f_1,\dots, f_k)\in \C[x]_0^k$, 
  \begin{align*}
\Phi_1(f_1,\dots, f_k) 
&= \frac{1}{k!} \sum_{l=1}^k (-1)^{k-l} \sum_{\substack{J \subset [k]\\ |J| = l}} \Phi_{1}(f_J,\dots, f_J)\\
&=\frac{1}{k!} \sum_{l=1}^k (-1)^{k-l} \sum_{\substack{J \subset [k]\\ |J| = l}} \Phi_{2}(f_J,\dots, f_J)\\
&=\Phi_2(f_1,\dots, f_k),
\end{align*}
where $f_J=\sum_{j\in J} f_j \in \mathbb{C}[x]_0$.
\end{proof}

\subsection{The degree and leading coefficient of polynomial $s_k[c_{m_1},\dots,c_{m_k}](z)$}

As a special family of polynomials, we take 
$$
c_m (x) := \binom{x}{m},
$$ 
for $m\in\N$. Then the general cases are induced from them by multi-linearity of $(f_1,\dots, f_k)\mapsto s_k[f_1,\dots, f_k](z)$; see Section \ref{sec: NewComb}.

Consider $r_k[c_{m_1},\dots, c_{m_k}](z)$ for positive integers $\{m_i\}_{i=1}^k$.
By Equation \eqref{eq:def_of_r}, its coefficient is
\begin{equation}\label{eq:r_k^{(n)}}
  \begin{split}
    r_k^{(n)}[c_{m_1},\dots, c_{m_k}](0) &= r_1^{(n)}[c_{m_1}\cdots c_{m_k}](0)\\
    &= \sum_{l=1}^n \binom{n}{l}(-1)^{n-l}c_{m_1}(l) \cdots c_{m_k}(l)\\
    &= \sum_{l=1}^n \binom{n}{l}(-1)^{n-l}\binom{l}{m_1} \cdots \binom{l}{m_k},
  \end{split}
\end{equation}
where $ r_k^{(n)}[c_{m_1},\dots, c_{m_k}](0) := \frac{d^n}{dz^n}r_k[c_{m_1},\dots, c_{m_k}](z)|_{z=0}$. This value has a combinatorial interpretation as follows.

\begin{lemma} \label{prop:coefficients_r}
  It holds that
  \begin{equation*}
    r_k^{(n)}[c_{m_1},\dots, c_{m_k}](0) = \#R_{(m_1,\dots,m_k)}^{(n)},
  \end{equation*}
  where $R_{(m_1,\dots,m_k)}^{(n)} := \{(W_1,\dots,W_k) \in \mathcal{B}(n)^k \ |\ |W_i| = m_i,\; \bigcup_{i=1}^k W_i = [n] \}$.
  In particular,
  \begin{enumerate}[\rm (1)]
    \item $\deg (r_k[c_{m_1},\dots, c_{m_k}]) = \sum_{i=1}^{k}m_i =: M$,
    \item $\lead (r_k[c_{m_1},\dots, c_{m_k}]) = \dfrac{1}{M!} \cdot r_k^{(M)}[c_{m_1},\dots, c_{m_k}]  {(0)} = \dfrac{1}{\prod_{i=1}^k m_i!}$.
  \end{enumerate}
\end{lemma}

\begin{proof}
Equation \eqref{eq:r_k^{(n)}} implies
\begin{align*}
r_k^{(n)}[c_{m_1},\dots, c_{m_k}](0) &= \sum_{l=1}^n \sum_{\substack{V\in \mathcal{B}(n)\\ |V|=l}} \sum_{\substack{W_1,\dots, W_k \subset V \\ |W_1|=m_1, \dots, |W_k|=m_k }} (-1)^{n-l}.
\end{align*}
Define $W=\bigcup_{i=1}^k W_i$. Exchanging the order of summations shows 
\begin{align*}
r_k^{(n)}[c_{m_1},\dots, c_{m_k}](0) &= \sum_{\substack{W_1,\dots, W_k \in \mathcal{B}(n) \\ |W_1|=m_1, \dots, |W_k|=m_k }}  \sum_{l=1}^n\left( \sum_{\substack{W \subset V\\ |V|=l}} (-1)^{n-l} \right)\\
&=\sum_{\substack{W_1,\dots, W_k \in \mathcal{B}(n) \\ |W_1|=m_1, \dots, |W_k|=m_k }}  \sum_{W \subset V} (-1)^{n-|V|}\\
&=\sum_{\substack{W_1,\dots, W_k \in \mathcal{B}(n) \\ |W_1|=m_1, \dots, |W_k|=m_k }}  \delta_{W, [n]}\\
&=\#R_{(m_1,\dots,m_k)}^{(n)},
\end{align*}
where the third equation follows from Lemma \ref{prop:Mobius on posets} and Example \ref{ex:Mobius_poset} (1).
The properties (1) and (2) easily follow from the definition of the set $R_{(m_1,\dots,m_k)}^{(n)}$.
\end{proof}

Similar to the polynomial $r_k[c_{m_1},\dots, c_{m_k}](z)$,
we can give a combinatorial interpretation to the coefficients of $s_k[c_{m_1},\dots, c_{m_k}](z)$ which reflects the intrinsic decomposition of $R_{(m_1,\dots,m_k)}^{(n)}$.
Let us introduce a few concepts to explain it. 
\begin{itemize}
\item Define the natural map $\tau: \mathcal{B}(n)\setminus \emptyset \rightarrow \calP(n)$ as
$$
\tau(W) = \{W\} \cup \{\{j\} \ |\ j \notin W\}
$$
for all non-empty subsets $W \subset [n]$;
\item An {\it ordered partition} of $[n]$ is a tuple $\mathbf{ V} = (V_1,\dots, V_l)$ such that $\{V_1,\dots, V_l\}\in \calP(n)$.
\item There is a canonical map that makes a partition $\pi = \{V_1,\dots, V_l\} \in \calP(n)$ correspond to an ordered partition $\overrightarrow{\pi} = (V_1,\dots, V_l)$ such that $1\in V_1$ and $V_i$ contains the minimum number in $[n]\setminus (V_1\cup \cdots \cup V_{i-1})$ for $i\ge 2$.
We call $\overrightarrow{\pi}$ the {\it natural} ordered partition of $[n]$ associated with $\pi \in \calP(n)$.
Clearly, this map is a one-to-one correspondence.
\end{itemize}

 {
The following concepts are crucial for interpreting the coefficients of $s_k[c_{m_1},\dots, c_{m_k}](z)$ combiniatorally. .}

\begin{definition} \label{def:S}
  Let $W = (W_1, \dots, W_k) \in R_{(m_1,\dots,m_k)}^{(n)}$.
  \begin{enumerate}[\rm (1)]
    \item A tuple $W$ is said to be {\it separable} if $\vee_{i=1}^k \tau(W_i) \neq 1_n \in \calP(n)$.
    \item A tuple $W$ is said to be {\it essential} if $W$ is not separable.
  \end{enumerate}
 Let $S_{(m_1,\dots,m_k)}^{(n)}$ denote the set of essential tuples in $R_{(m_1,\dots,m_k)}^{(n)}$. 
\end{definition}

The set $R_{(m_1,\dots, m_k)}^{(n)}$ can be decomposed into its components which consist of essential tuples.
One can see that, for $(W_1,\dots, W_k)\in R_{(m_1,\dots, m_k)}^{(n)}$, there are uniquely a natural ordered partition $\overrightarrow{\pi} = (V_1,\dots, V_l)$ of $[k]$ and an ordered partition $\mathbf{ U} =(U_1,\dots, U_l)$ of $[n]$ with $\bigvee_{r\in V_j} \tau(W_r) =\tau\left(U_j \right)$ for each $1\le j \le l$. 

\begin{example}
Take $(W_1,W_2,W_3) \in R_{(2,2,2)}^{(5)}$, where
\[
W_1=\{1,4\}, \quad W_2 =\{2,4\} \quad \text{and} \quad W_3=\{3,5\}.
\]
Then $\overrightarrow{\pi}=(V_1,V_2)=(\{1,2\},\{3\})$ and $\mathbf{ U}=(U_1,U_2)=(\{1,2,4\},\{3,5\})$. It is clear that $\mathbf{ U}=\left(\bigcup_{r\in V_1}W_r,\ \bigcup_{r\in V_2}W_r\right)$. Moreover, $\tau(W_1)\lor \tau(W_2)=\{\{1,2,4\},\{3\},\{5\}\}=\tau(U_1)$ and $\tau(W_3)=\tau(U_2)$.
\end{example}

Define 
\[
  S_{(m_1,\dots, m_k)}^{{\overrightarrow{\pi}, \mathbf{ U}}}
  =\left\{ (W_1,\dots,W_k) \in R_{(m_1,\dots, m_k)}^{(n)} \;\middle|\;   
  \bigcup_{r \in V_j} W_r =U_j,\,
  \bigvee_{r\in V_j} \tau(W_r)= \tau\left( U_j\right)
  \right\}.
\]
for a natural ordered partition $\overrightarrow{\pi} =(V_1,\dots,V_l)$ of $[k]$, an ordered partition $\mathbf{ U} = (U_1, \dots, U_l)$ of $[n]$ and positive integers $\{m_i\}_{i=1}^k$.
Note that $S_{(m_1,\dots, m_k)}^{{\overrightarrow{\pi}, \mathbf{ U}}} $ is isomorphic to $S_{V_1}^{(|U_1|)}\times \cdots \times S_{V_l}^{(|U_l|)}$ where $S_V^{(i)}:=S_{(m_{r_1},\dots, m_{r_t})}^{(i)}$ when $V=\{r_1,\dots, r_t\}$.

A consequence of the above one-to-one correspondence is that
\begin{align*}
R_{(m_1,\dots, m_k)}^{(n)} &= \bigcup_{\overrightarrow{\pi}=(V_1,\dots, V_l)} \bigcup_{\mathbf{ U}=(U_1,\dots, U_l)} S_{(m_1,\dots, m_k)}^{{\overrightarrow{\pi}, \mathbf{ U}}}\\
&= \bigcup_{\overrightarrow{\pi}=(V_1,\dots, V_l)} \bigcup_{\substack{i_1,\dots, i_l \ge1\\ i_1+\cdots+ i_l=n}} \ \bigcup_{\substack{\mathbf{ U}=(U_1,\dots, U_l)\\ |U_j|=i_j}} S_{(m_1,\dots, m_k)}^{{\overrightarrow{\pi}, \mathbf{ U}}}
\end{align*}
and therefore 
\begin{align}\label{eq:R_count}
\# R_{(m_1,\dots, m_k)}^{(n)} = \sum_{\pi=\{V_1,\dots, V_l\}\in \calP(k)} \sum_{\substack{i_1,\dots, i_l\ge1\\ i_1+\cdots + i_l=n}}\binom{n}{i_1,\dots, i_l} \prod_{j=1}^l \# S_{V_j}^{(i_j)}.
\end{align}

The value $s_k^{(n)}[c_{m_1},\dots, c_{m_k}] (0)$ can be computed by counting the number of $S_{(m_1,\dots, m_k)}^{(n)}$.

\begin{proposition} \label{prop:s}
  It holds that
  \[
  s_k^{(n)}[c_{m_1},\dots, c_{m_k}] (0)= \#S_{(m_1,\dots,m_k)}^{(n)}.
  \]
  In particular, $s_k[c_{m_1},\dots, c_{m_k}](z)$ is a polynomial of degree $\sum_{i=1}^{k}m_i - (k-1)$.
\end{proposition}

\begin{proof}
  By using the induction, it is not difficult to see that there are no essential tuples in $R_{(m_1,\dots,m_k)}^{(n)}$ if $n > \sum_{i=1}^{k}m_i - (k-1)$, which means $\# S_{(m_1,\dots,m_k)}^{(n)} =0$.

  Define the polynomials 
  $$
  \tilde{s}_k[{m_1},\dots, {m_k}](z) = \sum_{n=1}^{\infty} \frac{\# S_{(m_1,\dots,m_k)}^{(n)}}{n!} z^n.
  $$
  For the conclusion, it suffices to show that $s_k[c_{m_1},\dots, c_{m_k}](z) = \tilde{s}_k[{m_1},\dots, {m_k}](z)$, which is equivalent to 
  \begin{equation}\label{eq:kome}
    r_k[c_{m_1},\dots, c_{m_k}](z) =  \sum_{\pi \in \calP(k)} \tilde{s}_\pi[{m_1},\dots, {m_k}](z)
  \end{equation}
  because of the M\"{o}bius inversion formula.
  By Lemma \ref{prop:coefficients_r}, Equation \eqref{eq:kome} is equivalent to
  \begin{equation*} 
    \#R_{(m_1,\dots,m_k)}^{(n)} = \sum_{\pi = \{V_1, \dots, V_l\} \in \calP(k)}\sum_{\substack{i_1,\dots,i_l \ge 1\\i_{1} + \cdots + i_{l}=n}} \binom{n}{i_1, \dots, i_l} \prod_{j=1}^l \# S_{V_j}^{(i_j)}.
  \end{equation*}
  for $n\in\N$, which is exactly identical to Equation \eqref{eq:R_count}. 
\end{proof}

Now, the last problem is to determine the leading coefficient of $s_k[c_{m_1},\dots, c_{m_k}](z)$, i.e., to count $S_{(m_1,\dots,m_k)}^{(M-(k-1))}$ where $M = \sum_{i=1}^k m_i $.
The main strategy is to use the mathematical induction, which requires a slight modification of $S_{(m_1,\dots,m_k)}^{(M-(k-1))}$.

\begin{definition} \label{def:T}
  Let $\{m_i\}_{i=1}^k$ and $\{l_i\}_{i=1}^{M-(k-1)}$ be sequences of positive integers and $L:=l_1+\cdots+ l_{M-(k-1)}$. Define
  \[
    T_{(m_1,\dots,m_k)}^{(l_1, \dots,l_{M-(k-1)})} := \left\{(W_1,\dots,W_k) \in \mathcal{B}(L)^{k} \;\middle|\;  |W_{i}| = m_{i},\, \bigvee_{i=1}^k \tau(W_i) \vee \widehat{0}_{M-(k-1)} = 1_L \right\},
  \]
  where $\widehat{0}_{M-(k-1)}:=\{ \{1,\dots, l_1\}, \{l_1+1,\dots, l_1+l_2\},\dots, \{\sum_{i=1}^{M-(k-1)-1}l_i+1, \dots, \sum_{i=1}^{M-(k-1)} l_i\}\}$. 
\end{definition}

It is clear from Definitions \ref{def:S} and \ref{def:T} that $S_{(m_1,\dots,m_k)}^{(M-(k-1))}$ is a specific case of $T_{(m_1,\dots,m_k)}^{(l_1, \dots,l_{M-(k-1)})}$.
That is,
\begin{equation} \label{eq:star}
  S_{(m_1,\dots,m_k)}^{(M-(k-1))} = T_{(m_1,\dots,m_k)}^{(1,\dots, 1)},
\end{equation}
where $(1,\dots, 1)$ is a $(M-(k-1))$-tuple which consists only of $1$.

\begin{example}\label{ex:smallcase_T}
  Let us look at examples for small $k$.
  \begin{itemize}
  \item For any positive integers $m_1$ and $l_1,\dots, l_{m_1}$, one has
\[
  \# T_{(m_1)}^{(l_1, \dots,l_{m_1})} = l_1 \cdots l_{m_1}.
\]
\item For any positive integers $m_1,m_2$ and $l_1,\dots, l_{m_1+m_2-1}$, one has
\[
  \# T^{(l_1,\dots, l_{m_1 + m_2 -1})}_{(m_1,m_2)} = \sum_{\substack{I\subset [m_1+m_2-1]\\|I| = m_1}} 
    \left(\prod_{i\in I} l_i\right) \# T^{(l'_1,\dots,l'_{m_2})}_{(m_2)},
\]
where $l'_1= \sum_{i\in I} l_i$ and $\{l'_2,\dots,l'_{m_2}\} = \{l_i \mid i \in [m_1+m_2-1] \setminus I\}$.
Thus,
\[
  \begin{split}
    \# T^{(l_1,\dots, l_{m_1 + m_2 -1})}_{(m_1,m_2)} &= \sum_{\substack{I\subset [m_1+m_2-1]\\|I| = m_1}}  \left(\prod_{i=1}^{m_1+m_2-1}l_i \right)\sum_{i\in I} l_i \\
    &= \left(\prod_{i=1}^{m_1+m_2-1}l_i\right) \binom{m_1+m_2-2}{m_1-1}\sum_{i=1}^{m_1+m_2-1}l_i.
  \end{split}
\]
\end{itemize}
\end{example}
Applying the counting technique used in the above example to the general case, we obtain the following results.
\begin{proposition}\label{prop:T}
Let $\{m_i\}_{i=1}^k$ and $\{l_j\}_{j=1}^{M-(k-1)}$ be sequences of positive integers where $M=\sum_{i=1}^km_i$. Then
  \begin{align}\label{eq:T}
  \# T^{(l_1,\dots, l_{M-(k-1)})}_{(m_1,\dots,m_k)} = \left(\prod_{i=1}^{M-(k-1)}l_i\right) \binom{M-k}{m_1-1, \dots, m_k-1}\left(\sum_{i=1}^{M - (k-1)}l_i \right)^{k-1}.
\end{align}
\end{proposition}
\begin{proof}
  We use the induction for $k$ as follows.
  Formula \eqref{eq:T} holds for $k=1$ because we mentioned in Example \ref{ex:smallcase_T}.
  Assume Formula \eqref{eq:T} holds up to $k-1$.
  Note that
  \begin{equation*}
  \# T^{(l_1,\dots, l_{M-(k-1)})}_{(m_1,\dots,m_k)} =\sum_{\substack{I\subset [M-(k-1)]\\|I| = m_1}} \left(\prod_{i\in I} l_i\right) \# T^{(l'_1,\dots,l'_{m_2+\cdots+m_k-(k-1)})}_{(m_2,\dots,m_k)},
  \end{equation*}
  where  $l'_1= \sum_{i\in I} l_i$ and $\{l'_2,\dots,l'_{m_2+\cdots+m_k-(k-1)}\} =\{ l_i \mid i\in [M-(k-1)] \setminus I$\}.
  Thus, by the induction hypothesis \eqref{eq:T},
  \begin{align*}
    &\# T^{(l_1,\dots, l_{M-(k-1)})}_{(m_1,\dots,m_k)} \\
    &= \sum_{\substack{I\subset [M-(k-1)]\\|I| = m_1}} \left(\prod_{i=1}^{M - (k-1)}l_i\right) \left(\sum_{i\in I} l_i\right) \binom{m_2+\cdots+m_k-(k-1)}{m_2-1, \dots, m_k-1}\left(\sum_{i=1}^{M - (k-1)}l_i \right)^{k-2}\\
    &= \left(\prod_{i=1}^{M - (k-1)}l_i\right) \binom{m_2+\cdots+m_k-(k-1)}{m_2-1, \dots, m_k-1} \left(\sum_{i=1}^{M - (k-1)}l_i \right)^{k-2} \sum_{\substack{I\subset [M-(k-1)]\\|I| = m_1}}  \left(\sum_{i\in I} l_i\right) \\
    &= \left(\prod_{i=1}^{M - (k-1)}l_i\right) \binom{m_2+\cdots+m_k-(k-1)}{m_2-1, \dots, m_k-1} \left(\sum_{i=1}^{M - (k-1)}l_i \right)^{k-2}\binom{M-k}{m_1-1} \left(\sum_{i=1}^{M - (k-1)}l_i \right) \\
    &= \left(\prod_{i=1}^{M - (k-1)}l_i\right) \binom{M-k}{m_1-1, \dots, m_k-1}\left(\sum_{i=1}^{M - (k-1)}l_i \right)^{k-1}.
  \end{align*}
\end{proof}

\begin{corollary} \label{cor:lead_s}
Let $\{m_i\}_{i=1}^k$ be a sequence of positive integers and $M=\sum_{i=1}^k m_i$. Then
  \[
    \lead (s_k[c_{m_1}, \dots, c_{m_k}]) = \frac{(M-(k-1))^{k-2}}{\prod_{i=1}^k (m_i-1)!}.
  \]
\end{corollary}
\begin{proof}
Recall that $\deg( s_k[c_{m_1},\cdots, c_{m_k}])=M-(k-1)$ and $s_k^{(M-(k-1))}[c_{m_1}, \dots, c_{m_k}](0) = \# S_{(m_1,\dots,m_k)}^{(M - (k-1))}$ by Proposition \ref{prop:s}.
Also, Equation \eqref{eq:star} and Proposition \ref{prop:T} imply that
 \[
  \begin{split}
    \# S_{(m_1,\dots,m_k)}^{(M - (k-1))} &= \# T_{(m_1,\dots,m_k)}^{(1,\dots,1)}\\
    &=\binom{M-k}{m_1-1, \dots, m_k-1}(M-(k-1))^{k-1}\\
    &=\frac{(M-(k-1))!}{\prod_{i=1}^k (m_i-1)!}\cdot (M-(k-1))^{k-2}.
  \end{split}
  \]
\end{proof}

\subsection{New combinatorial formula} \label{sec: NewComb}
\begin{theorem}[see Theorem \ref{main1}]\label{thm:Combinatorics}
  Suppose that $f_i \in \C[x]_0$ with $\deg (f_i)=m_i$ for each $1 \le i \le k$, and let $M=\sum_{i=1}^k m_i$.  Then we have
  \begin{align*}
  \sum_{\pi \in \mathcal{P}(n)} \prod_{i=1}^k\left(  \sum_{V\in \pi} f_i(|V|) \right)\mu_n(\pi,1_n)=\begin{cases}
  (n-1)!  n^{k-1} \prod_{i=1}^km_i\lead (f_i), & n=M-(k-1), \\
  0, & n>M-(k -1).
  \end{cases}
  \end{align*}
\end{theorem}
\begin{proof}
  By Proposition \ref{prop:multiple}, the statement is equivalent to the following:
  \begin{itemize}
    \item $\deg (s_k[f_1,\dots,f_k]) = M - (k-1)$;
    \item $\lead (s_k[f_1,\dots,f_k]) = (M - (k-1))^{k-2}\cdot \prod_{i=1}^k m_i\lead (f_i)  $.
  \end{itemize}
  Because the family of polynomials $\{c_m\}_{m\in\N}$ is a basis of $\C[x]_0$, the polynomials $\{f_i\}_{i=1}^k$ can be uniquely expressed as linear combinations of $\{c_m\}_{m\in\N}$:
  \[
    f_i(x) = \sum_{j=1}^{m_i} a_{j}^{(i)} c_j(x)
  \]
  for $1 \le i \le k$.
  Here, note that $c_m(x)$ is the polynomial of degree $m$ and the leading coefficient $1/m!$, and hence $a^{(i)}_{m_i} = m_i! \lead (f_i)$.
  Next, by  {Lemma} \ref{prop:multi-linear},
  \[
    s_k[f_1,\dots,f_k](z) =\sum_{j_1=1}^{m_1} \cdots \sum_{j_k=1}^{m_k} a_{j_1}^{(1)} \cdots a_{j_k}^{(k)} s_k[c_{j_1},\dots, c_{j_k}](z).
  \]
  Then, by Proposition \ref{prop:s},
  \[
    \begin{split}
      \deg (s_k[f_1,\dots, f_k]) &= \deg (s_k[c_{m_1},\dots, c_{m_k}]) \\
      &= M - (k-1).
    \end{split}
  \]
  Hence, by Corollary \ref{cor:lead_s},
  \[
    \begin{split}
      \lead (s_k[f_1,\dots, f_k]) &= a^{(1)}_{m_1} \cdots a^{(k)}_{m_k}\lead (s_k[c_{m_1},\dots, c_{m_k}]) \\
      &= \lead (s_k[c_{m_1},\dots, c_{m_k}]) \prod_{i=1}^k m_i ! \lead (f_i)\\
      &=\frac{(M - (k-1))^{k-2}}{\prod_{i=1}^k (m_i-1)!}  \prod_{i=1}^k m_i ! \lead( f_i)\\
      &= (M - (k-1))^{k-2} \prod_{i=1}^km_i\lead (f_i).
    \end{split}
  \]
\end{proof}


\begin{corollary}\label{ex:Combthm}
For $n,k \in \N \; (n \ge k+1)$,
\begin{align} \label{eq:chV/2}
\sum_{\pi\in \calP(n)} \left(\sum_{V\in \pi} \binom{|V|}{2} \right)^k \mu_n(\pi,1_n) =\begin{cases}
(n-1)!n^{k-1}, & n=k+1,\\
0, & n>k+1,
\end{cases}
\end{align}
and
\begin{align} \label{eq:V^2}
\sum_{\pi\in \calP(n)} \left( \sum_{V\in \pi} |V|^2 \right)^k \mu_n(\pi,1_n) =\begin{cases}
2^k(n-1)!n^{k-1}, & n=k+1,\\
0, & n>k+1.
\end{cases}
\end{align}
These results will be used in the later sections.
\end{corollary}

Finally, we conclude this section by noting that the combination of Example \ref{ex:coffee}  and Theorem \ref{thm:Combinatorics} leads to the following results as byproducts.
\begin{corollary}\label{cor:counting}
  For positive integers $\{m_i\}_{i=1}^k$, define $M = \sum_{i=1}^k m_i$ and $\widehat{0}_k = \{\{1, \dots, m_1\}, \{m_1+1, \dots, m_1 + m_2\}, \dots, \{M - m_k+1, \dots, M\}\}$.
  Then
  \[
  \# \{\sigma \in \calP(M) \ |\ \sigma \vee \widehat{0}_k = 1_M, |\sigma| = M - (k-1)\} = (M-(k-1))^{k-2} \prod_{i=1}^k m_i.
  \]
\end{corollary}

\begin{remark}
  For the cases $n < M -(k-1)$ in Theorem \ref{thm:Combinatorics}, we could not obtain an explicit formula in general although the cases are not important for the later sections.
  A key to the desired formula is counting 
\[
    \# \{\sigma \in \mathcal{P}(M) \ |\ \sigma \vee \widehat{0}_k = 1_M, |\sigma| = n\},
\]
which is equivalent to determine $s_k^{(n)}[g_{m_1}, \dots, g_{m_k}]$, where $g_m(x) = x^m$ for $m\in\mathbb{N}$.
\end{remark}

\section{Finite free analogue of Sakuma-Yoshida's limit theorem}

In this section, we establish the finite free analogue of the limit theorem by Sakuma and Yoshida \cite{SY}. More precisely, our purpose in this section is to investigate the limit behavior of the sequence of finite free cumulants of 
\[
  D_{1/m}((p_d^{\boxtimes_dm})^{\boxplus_dm}),
\]
as $m\rightarrow\infty$ for $p_d \in \monP(d)$ with nonnegative roots, such that $\kappa_1^{(d)}(p_d)=1$.
Recall that, Sakuma and Yoshida investigated the asymptotic expansion of S-transform, in contrast to that, Arizmendi and Vargas \cite{AV} gave another proof by focusing on the combinatorial structure of the non-crossing partitions.
Here, we will take the latter approach, i.e., the convergence of finite free cumulants.

Suppose first that degree $d$ is fixed.
According to Equation \eqref{eq:kappa_D} and Proposition \ref{thm:newlimit}, we have
\[
  \kappa_n^{(d)} \left(D_{1/m}((p_d^{\boxtimes_dm})^{\boxplus_dm})\right)=\frac{1}{m^{n-1}} \kappa_n^{(d)}(p_d^{\boxtimes_dm})\xrightarrow{m\rightarrow\infty} 0
\]
for $n=2, \dots, d$, which implies $D_{1/m}((p_d^{\boxtimes_dm})^{\boxplus_dm})(x)\rightarrow (x-1)^d$ as $m\rightarrow\infty$.
Hence this is not an interesting result.
In order to observe non-trivial limits of finite free cumulant, we consider the following two situations of $m\rightarrow\infty$ with (i) $m/d\rightarrow t$ for some $t>0$, and (ii) $m/d\rightarrow0$. 

In the following, we consider $p_d \in \monP(d)$ with nonnegative roots.
Additionally, we assume that
\begin{enumerate}
\item[(A-1)] $\kappa_1^{(d)}(p_d)=1$, that is, $\widetilde{a}_1(p_d)=1$;
\item[(A-2)] there exists a probability measure $\mu$ with compact support such that $\mu\llbracket p_d \rrbracket\xrightarrow{w} \mu$ as $d\rightarrow\infty$.
\end{enumerate}
We define 
$$
e_n(t, \mu) := \exp\left(-t\binom{n}{2}\kappa_2(\mu)\right), \qquad t >0
$$
for $n \in \N$. Note that, for $\pi \in \calP(n)$,
\[
e_\pi (t,\mu)=\prod_{V\in \pi} \exp\left(-t\binom{|V|}{2} \kappa_2(\mu)\right)=\exp\left(-t \sum_{V\in \pi}\binom{|V|}{2} \kappa_2(\mu)\right).
\]

\subsection{Case of $m/d\rightarrow t$ for some $t>0$}

In this section, we consider the case when a ratio of a number $m$ of finite free multiplicative convolution and degree $d$ of polynomial converges to some $t>0$ as $d,m\rightarrow \infty$, that is, $m/d \rightarrow t$.
Our conclusion in this section is the following.

\begin{proposition}\label{thm:m=td}
  Let us consider $p_d\in\mathbb{P}_{\mathrm{mon}}(d)$ satisfying (A-1) and (A-2). For $n \in \N$,
  \[
    \kappa_n^{(d)} \left(D_{1/m}((p_d^{\boxtimes_dm})^{\boxplus_dm})\right) \rightarrow \frac{(-1)^{n-1}}{t^{n-1}(n-1)!} \sum_{\pi \in \calP(n)} e_\pi (t,\mu) \mu_n(\pi,1_n),
  \]
  as $d,m\rightarrow \infty$ with $m/d\rightarrow t$ for some $t>0$.
\end{proposition}
\begin{proof}
By Lemma \ref{lem:finite_free_multi}, we have
\begin{align*}
  \widetilde{a}_n(p_d^{\boxtimes_d m})  =\widetilde{a}_n(p_d)^m  = \left( \kappa_1^{(d)}(p_d)^n - \binom{n}{2}\frac{\kappa_1^{(d)}(p_d)^{n-2}\kappa_2^{(d)}(p_d)}{d} + O(d^{-2}) \right)^m.
\end{align*}
Since $\kappa_1^{(d)}(p_d) = 1$ by (A-1) and $\kappa_2^{(d)}(p_d)\rightarrow \kappa_2(\mu)$ as $d\rightarrow\infty$ by (A-2), we then obtain
\begin{equation} \label{eq:apple}
\begin{split}
  \widetilde{a}_n(p_{d}^{\boxtimes_d m}) 
  &= \left( 1 - \binom{n}{2}\frac{\kappa_2^{(d)}(p_d)}{d} + O(d^{-2}) \right)^m \\
  &=\left\{\left( 1 - \binom{n}{2}\frac{\kappa_2^{(d)}(p_d)}{d} + O(d^{-2}) \right)^{d}\right\}^{\frac{m}{d}}\\
  &\rightarrow \exp\left(-t \binom{n}{2} \kappa_2(\mu) \right),
  \end{split}
\end{equation}
as $m \rightarrow \infty$ with $m/d\rightarrow t$.
\end{proof}

Let us denote by $\{\kappa_n(t,\mu)\}_{n\geq1}$ the above limit of finite free cumulants.

\begin{corollary}
For $n\in \N$, we have
\[
\lim_{t\rightarrow 0^+} \kappa_n(t,\mu)= \frac{(\kappa_2(\mu) n)^{n-1}}{n!}.
\]
\end{corollary}
\begin{proof}
It is easy to see that $\lim_{t\rightarrow 0^+} e_\pi(t,\mu)=1$. Note that 
  \[
  \sum_{\pi\in \calP(n)} \mu_n(\pi, 1_n)=0
  \]
  and 
  \[
  \frac{\partial^{k}}{\partial t^{k}} e_\pi (t,\mu)= \left(- \sum_{V\in \pi} \binom{|V|}{2} \kappa_2(\mu) \right)^k e_\pi (t,\mu).
  \]
  Due to L'H\^{o}pital's theorem, Proposition \ref{thm:m=td} and Equation \eqref{eq:chV/2}, the following result is obtained:
  \begin{align*}
  \lim_{t\rightarrow 0^+} \kappa_n(t,\mu) &=\frac{(-1)^{n-1}}{(n-1)!} \sum_{\pi \in \calP(n)}  \left\{\lim_{t\rightarrow 0^+} \frac{\left( -\sum_{V\in \pi} \binom{|V|}{2} \kappa_2(\mu) \right)^{n-1} e_\pi (t,\mu)}{(n-1)!} \right\}\mu_n(\pi,1_n)\\
  &=\frac{\kappa_2(\mu)^{n-1}}{((n-1)!)^2} \sum_{\pi \in \calP(n)} \left( \sum_{V\in \pi} \binom{|V|}{2}\right)^{n-1} \mu_n(\pi,1_n)\\
  &=\frac{(\kappa_2(\mu) n)^{n-1}}{n!}.
  \end{align*}
\end{proof}

\begin{example} 
For simplicity, we assume that $\kappa_2^{(d)}(p_d)=1$. Then it also satisfies that $\kappa_2(\mu)=1$. Then the first four cumulants are computed as follows.
  \begin{itemize}
   \item $\kappa_1(t,\mu) = 1$ and $\lim_{t\rightarrow 0^+} \kappa_1(t,\mu)=1$.
   \item $\kappa_2(t,\mu) = \dfrac{e^{t}-1}{t}$ and $\lim_{t\rightarrow 0^+} \kappa_2(t,\mu)=1$.
    \item $\kappa_3(t,\mu) = \dfrac{e^{3t} - 3e^t +2}{2t^2}$ and $\lim_{t\rightarrow 0^+} \kappa_3(t,\mu)=\dfrac{3}{2}$.
    \item $\kappa_4(t,\mu) =\dfrac{e^{6t} - 4 e^{3 t}  - 3 e^{2 t} + 12 e^t    - 6}{6t^3}$ and $\lim_{t\rightarrow 0^+} \kappa_4(t,\mu)=\dfrac{8}{3}$.
  \end{itemize}
  \end{example}


\subsection{Case of $m/d \to 0$ as $d,m\rightarrow\infty$}

For the proof of Theorem \ref{thm:mainresult}, we prepare a few concepts.

\begin{definition}
  Let $\C[[x,y]]$ denote the ring of formal power series of $x$ and $y$:
  \[
    \C[[x,y]] := \left\{f(x,y) = \sum_{k,l=0}^\infty a_{kl}x^k y^l \;\middle|\; a_{kl} \in \C \right\}.
  \]
  For any $f(x,y) \in \C[[x,y]]$, there uniquely exist $\{q_l[f](x)\}_{l=0}^\infty \subset \C[[x]]$ such that
  \[
    f(x,y) = \sum_{l=0}^\infty q_l[f](x)y^l
  \]
  since $\C[[x,y]] \cong (\C[[x]])[[y]]$.
  We define a subset $\SUB$ of $\C[[x,y]]$ as 
  \[
    \SUB := \left\{f(x,y) =  \sum_{l=0}^\infty q_l[f](x)y^l \in \C[[x,y]] \;\middle|\; \deg q_l[f] \le l  \right\}.
  \]
\end{definition}

\begin{lemma}\label{lem:subring}
  The subset $\SUB$ is closed under addition and multiplication, i.e. $\SUB$ is a subring of $\C[[x,y]]$.
\end{lemma}

\begin{proof}
  Indeed, by definition,
  $$q_l[f + g] = q_l[f] + q_l[g]$$ and 
  $$q_l[fg] = \sum_{i+j=l} q_i[f] q_j[g]$$ 
  for $f,g \in \SUB$, $l \ge 0$.
  In particular, $$\lead(q_l[fg]) = \sum_{i+j=l} \lead(q_i[f]) \lead (q_j[g]).$$
\end{proof}

\begin{lemma} \label{lem:power}
  For a formal power series $f(x) = 1 + \sum_{k=1}^\infty \frac{a_k}{k!} x^k \in \C[[x]]$, consider
  \[
    F(m,x):= f^m(x) = \left(1+\sum_{k=1}^d \frac{a_k}{k!} x^k\right)^m = \sum_{l=0}^\infty q_l[F](m) x^l \in \C[[m,x]].
  \]
  Then $F \in \SUB$, $\deg (q_l[F]) = l$, and $\lead (q_l[F]) = a_1^l/l!$ for $l \ge 0$.
\end{lemma}

\begin{proof}
  Clearly $q_0[F](m) = F(0,m) = f^m(0) = 1$ and 
  \[
    q_1[F](m) = \left.\frac{\partial}{\partial x}\right|_{x=0}F(m,x) = nf^{m-1}(0)f'(0) = m a_1.
  \]

  Define $g(x) = \sum_{k=1}^\infty \frac{a_k}{k!} x^k$ and then $f(x) = 1 + g(x)$.
  Thus,
  \begin{align*}
    F(m,x) &= (1 + g(x))^m\\
    &= \sum_{k=0}^\infty \binom{m}{k} g(x)^k
  \end{align*}
  and so
  \[
    q_l[F](m) = \frac{1}{l!}\left.\frac{\partial^l}{\partial x^l}\right|_{x=0}F(m,x) = \sum_{k=1}^l \binom{m}{k} \frac{1}{l!}\left.\frac{\partial^l}{\partial x^l}\right|_{x=0} g(x)^k.
  \]
  Hence, $\deg q_l[F] = l$ and 
  \[
    \lead q_l[F] = \frac{1}{(l!)^2}\left.\frac{\partial^l}{\partial x^l}\right|_{x=0} g(x)^l = \frac{g'(0)^l}{l!} = \frac{a_1^l}{l!}.
  \]
\end{proof}

\begin{theorem}[See Theorem \ref{main3}]\label{thm:mainresult}
Let us consider $p_d\in \mathbb{P}_{\mathrm{mon}}(d)$ satisfying (A-1) and (A-2). It satisfies that
$$
  \kappa_n^{(d)} \left(D_{1/m}((p_d^{\boxtimes_dm})^{\boxplus_dm})\right) \to \frac{(\kappa_2(\mu)n)^{n-1}}{n!}, \quad n \in \N
 $$
as $d, m\rightarrow\infty$ with $m/d \to 0$.
\end{theorem}

\begin{proof}
By Proposition \ref{prop:cumulant_formula},
\begin{align*}
  \kappa_n^{(d)}(p_d^{\boxtimes_dm})
   & =\frac{(-d)^{n-1}}{(n-1)!} \sum_{\pi \in \calP(n)} \widetilde{a}_\pi(p_d)^m \mu_n(\pi,1_n) \\
   & =\frac{(-d)^{n-1}}{(n-1)!}\sum_{\substack{ \sigma_1,\dots, \sigma_m \in \calP(n)         \\ \sigma_1\lor \cdots\lor \sigma_m=1_n}} \left( \prod_{i=1}^m d^{|\sigma_i|-n} \mu_n(0_n,\sigma_i) \kappa_{\sigma_i}^{(d)}(p_d)\right)\\
   & =:\frac{(-d)^{n-1}}{(n-1)!}\sum_{l=0}^\infty \frac{q_l(m)}{d^l},
\end{align*}
where, on the third line, we define the polynomials $\{q_l(m)\}_{l=0}^\infty$ as the coefficients of the expansion with respect to $d^{-1}$.
We immediately know that $q_l(m) = 0 \; (l=0, \dots, n-2)$ from the second line. One can verify that the degree of $q_l(m)$ is less than or equal to $l$.
Look at the first line and \eqref{eq:apple}, and then consider the coefficients of the expansion of $\widetilde{a}_n(p_d)^m \; (n\in \N)$ with respect to $d^{-1}$.
Then it follows that its degree is less than or equal to the degree of its denominator by Lemma \ref{lem:power}.
Also, it is true about $\widetilde{a}_\pi(p_d)^m$ for $\pi \in \calP(n)$ by Lemma \ref{lem:subring}.

Then we show the leading coefficient of $q_{n-1}(m)$ equals $(-\kappa_2(p_d))^{n-1}n^{n-2}$.
The leading coefficient of the numerator polynomial of $\widetilde{a}_\pi(p_d)^m$ for $d^{n-1}$
is computed as
\[
  \sum_{\substack{l_1, \dots, l_{|\pi|} \ge 0 \\ l_1 + \cdots +l_{|\pi|}=n-1}} \frac{1}{l_1! \cdots l_{|\pi|}!}  \left(-\kappa_2^{(d)}(p_d)\binom{|V_1|}{2}\right)^{l_1} \cdots \left(-\kappa_2^{(d)}(p_d)\binom{|V_{|\pi|}|}{2}\right)^{l_{|\pi|}}
\]
by Lemma \ref{lem:subring}.
Hence the leading coefficient of $q_{n-1}(m)$ equals to 
\begin{align*}
(-\kappa_2^{(d)}(p_d))^{n-1}&  \sum_{\substack{\pi \in \calP(n)\\ \pi = \{V_1, \dots, V_r\}}} \left\{\sum_{\substack{l_1, \dots, l_{r} \ge 0 \\ l_1 + \cdots +l_{r}=n-1}} \frac{1}{l_1! \cdots l_{r}!}  \binom{|V_1|}{2}^{l_1} \cdots \binom{|V_{r}|}{2}^{l_{r}}\right\} \mu_n(\pi,1_n)\\
&=\frac{(-\kappa_2^{(d)}(p_d))^{n-1}}{(n-1)!} \sum_{\pi \in \calP(n)} \left(\sum_{V \in \pi} \binom{|V|}{2} \right)^{n-1} \mu_n(\pi,1_n)\\
&=(-\kappa_2^{(d)}(p_d))^{n-1} n^{n-2},
\end{align*}
where the last equality holds due to  Equation \eqref{eq:chV/2}.
Finally, we obtain
\begin{align*}
  \kappa_n^{(d)} \left(D_{1/m}((p_d^{\boxtimes_dm})^{\boxplus_dm})\right) 
  &= \frac{(-d)^{n-1}}{m^{n-1}(n-1)!}\sum_{l=0}^\infty \frac{q_l(m)}{d^l}\\
  & = \frac{(\kappa_2^{(d)}(p_d) n)^{n-1}}{n!} + O\biggl(\frac{1}{m}\biggr) + O\biggl(\frac{m}{d} \biggr).
\end{align*}
\end{proof}

\section{Central limit theorem for finite free multiplicative convolution of polynomials with nonnegative roots}

In this section, we prove the law of large numbers (LLN) and the central limit theorem (CLT) for finite free multiplicative convolution of polynomials with nonnegative roots.  Moreover, we investigate a relation between the limit polynomial from the CLT and the multiplicative free semicircular distribution on $[0,\infty)$.

\subsection{Law of large numbers}

In \cite{FU}, the LLN of finite free multiplicative convolution of polynomials with nonnegative roots was established.
That is, if $p$ is a polynomial of degree $d$ with nonnegative roots, then the limit polynomial of $\phi_{1/m}(p^{\boxtimes_d m})$ as $m\rightarrow\infty$, was investigated, where $\phi_\alpha$ was defined in \eqref{eq:phi_a}.
Recall that its limit polynomial is nontrivial and expected to be closely related to results in \cite{HM}.
Here, we are interested in another type of the LLN for $\boxtimes_d$, that is, the limit behavior of $\phi_{1/m}(p)^{\boxtimes_d m}$ (is not equal to $\phi_{1/m}(p^{\boxtimes_dm})$) as $m\rightarrow\infty$.
However, it was not studied in \cite{FU}. In this section, we show that the limit roots of $\phi_{1/m}(p)^{\boxtimes_d m}$ concentrate on a point as $m\rightarrow\infty$. 


\begin{proposition}
Let $p(x)=\prod_{k=1}^d(x-e^{\theta_k})$. Assume that
\[
\frac{1}{d}\sum_{k=1}^d \theta_k =\alpha.
\]
Then
\[
\lim_{m\rightarrow\infty } \phi_{1/m}(p)^{\boxtimes_d m}(x)=(x-e^{\alpha})^d.
\]
\end{proposition}
\begin{proof}
A simple computation shows that
\begin{align*}
\tilde{a}_k(\phi_{1/m}(p))&=\binom{d}{k}^{-1} \sum_{1\le j_1<\cdots < j_k \le d} \exp\left(\frac{1}{m}(\theta_{j_1}+\cdots \theta_{j_k})\right)\\
&=\binom{d}{k}^{-1} \left\{\binom{d}{k}+\frac{1}{m}\sum_{1\le j_1<\cdots < j_k \le d}(\theta_{j_1}+\cdots \theta_{j_k})+O(m^{-2})\right\}\\
&=\binom{d}{k}^{-1}\left\{ \binom{d}{k}+ \frac{1}{m}\binom{d-1}{k-1}\sum_{i=1}^d\theta_i +O(m^{-2})\right\}\\
&=1+\frac{1}{m}\cdot \alpha k+O(m^{-2}).
\end{align*}
This implies that 
\[
\lim_{m\rightarrow\infty} \tilde{a}_k(\phi_{1/m}(p))^m=e^{\alpha k},
\]
and therefore
\[
\lim_{m\rightarrow\infty} \phi_{1/m}(p)^{\boxtimes_dm}(x)=\sum_{k=0}^d (-1)^k \binom{d}{k}e^{\alpha k} x^{d-k}=(x-e^{\alpha})^d.
\]
\end{proof}


\subsection{Central Limit Theorem and the multiplicative free semicircular distribution}
In this section, we investigate the CLT for $\boxtimes_d$ of polynomials with nonnegative roots and their limits as the degree goes to infinity.

\begin{theorem}[See Theorem \ref{main5} (1)]\label{thm:CLTpositive}
  Let $d \ge 2$.
  Suppose $p(x) = \prod_{k=1}^d(x-e^{\theta_k})$ such that 
  \begin{align*}
    \frac{1}{d}\sum_{k=1}^d \theta_k = 0 \qquad \text{and} \qquad    \frac{1}{d} \sum_{k=1}^d \theta_k^2 = \sigma^2.
  \end{align*}
  Then
  \begin{equation}
    \lim_{m \to \infty} \phi_{1/\sqrt{m}}(p)^{\boxtimes_d m}(x) = I_d\left(x;\frac{d \sigma^2}{d-1}\right),
  \end{equation}
  where $ I_d(x;t) = \sum_{k=0}^d (-1)^k \binom{d}{k} \exp(k(d-k)t/2d)x^{d-k}$ for $t \ge 0$.
\end{theorem}
\begin{proof}
  First, note that
  \begin{equation} \label{tukue}
    0 = \left(\sum_{k=1}^d \theta_k\right)^2 = d \sigma^2 + 2 \sum_{1 \le j_1 < j_2 \le d} \theta_{j_1} \theta_{j_2}
  \end{equation}
  by the assumptions.
  It follows that
  \[
    \begin{split}
      \widetilde{a}_k(\phi_{1/\sqrt{m}}(p)) &= \binom{d}{k}^{-1} \sum_{1 \le j_1 < \cdots < j_k \le d} \exp\left(\frac{1}{\sqrt{m}}(\theta_{j_1}+ \cdots + \theta_{j_k})\right)\\
      &= \binom{d}{k}^{-1} \left\{ \binom{d}{k} + \frac{1}{\sqrt{m}}\sum_{1 \le j_1 < \cdots < j_k \le d} (\theta_{j_1}+ \cdots + \theta_{j_k}) \right.\\
      &\hspace{20mm} \left. + \frac{1}{2m} \sum_{1 \le j_1 < \cdots < j_k \le d} (\theta_{j_1}+ \cdots + \theta_{j_k})^2 + O(m^{-\frac{3}{2}}) \right\}\\
      &= \binom{d}{k}^{-1} \left\{ \binom{d}{k} + \frac{1}{\sqrt{m}}\binom{d-1}{k-1}\sum_{k=1}^d \theta_k  \right.\\
      &\hspace{20mm} \left. + \frac{1}{2m} \left(\binom{d-1}{k-1}\sum_{k=1}^d \theta_k^2 +2\binom{d-2}{k-2}\sum_{1 \le j_1 < j_2 \le d} \theta_{j_1} \theta_{j_2} \right) + O(m^{-\frac{3}{2}}) \right\}\\
      &= \binom{d}{k}^{-1} \left\{\binom{d}{k} + \frac{d\sigma^2}{2m} \left(\binom{d-1}{k-1} - \binom{d-2}{k-2} \right) + O(m^{-\frac{3}{2}})\right\}\\
      &= 1 + \frac{k(d-k)}{2(d-1)m}\sigma^2 + O(m^{-\frac{3}{2}}),
    \end{split}
  \]
  where we used the assumptions and \eqref{tukue} on the third line.
  Thus, $\widetilde{a}_k(\phi_{1/\sqrt{m}}(p))^m$ goes to $\exp(k(d-k)\sigma^2/2(d-1))$ as $m \to \infty$ for $k=0, \dots, d$.
  It means $\lim_{m \to \infty} \phi_{1/\sqrt{m}}(p)^{\boxtimes_d m}(x) = I_d(x;d{\sigma^2}/{(d-1)})$.
\end{proof}


Finally, we show that the finite free cumulants of $I_d(x;t)$ converge to the free cumulants of the multiplicative free semicircular distribution $\lambda_t$ whose $n$-th moment $m_n(\lambda_t)$ is given by
\[
\exp\left(\frac{nt}{2}\right)\sum_{k=0}^{n-1} \frac{n^{k-1}}{k!} \binom{n}{k+1} t^k, \qquad t\ge 0.
\]
The distribution $\lambda_t$ was introduced in \cite[Proposition 5]{Biane95}. 


\begin{theorem}[See Theorem \ref{main5} (2)]\label{thm:FMS}
For $t\ge 0$, we have $\mu\llbracket I_d(x;t)\rrbracket \xrightarrow{w} \lambda_t$ as $d\rightarrow\infty$.
\end{theorem}
\begin{proof}
By the definition of $I_d(x ; t)$, we have
\begin{align*}
\widetilde{a}_k(I_d(x ; t))=\exp\left(\frac{tk(d-k)}{2d}\right), \qquad k=0,1, \dots, d,
\end{align*}
and therefore for $\pi \in \calP(n)$,
\begin{align*}
\widetilde{a}_\pi(I_d(x ; t))&=\prod_{V\in \pi} \exp\left(\frac{t|V|(d-|V|)}{2d}\right)\\
&=\exp\left\{ \frac{t}{2d} \left(\sum_{V\in \pi} d|V| -\sum_{V\in \pi}|V|^2\right) \right\}\\
&=\exp\left(\frac{tn}{2}\right) \exp\left(-\frac{t}{2d}\sum_{V\in \pi} |V|^2\right).
\end{align*}
Due to Equation \eqref{eq:V^2}, a straightforward computation shows that
\begin{align*}
\sum_{\pi\in \calP(n)} & \widetilde{a}_\pi(I_d(x ; t)) \mu_n(\pi,1_n)\\
&= \exp\left(\frac{tn}{2}\right) \sum_{\pi\in \calP(n)} \exp\left(-\frac{t}{2d}\sum_{V\in \pi} |V|^2\right)\mu_n (\pi,1_n) \\
&=\exp\left(\frac{tn}{2}\right) \sum_{\pi\in \calP(n)} \left\{ \sum_{k=0}^\infty \frac{1}{k!}\left(-\frac{t}{2d}\sum_{V\in \pi} |V|^2\right)^k \right\}\mu_n (\pi,1_n)\\
&=\exp\left(\frac{tn}{2}\right)\frac{1}{(n-1)!} \left(-\frac{t}{2d}\right)^{n-1} \left\{\sum_{\pi\in \calP(n)}\left(\sum_{V\in \pi}|V|^2\right)^{n-1} \mu_n (\pi,1_n)\right\}+O\left(\frac{1}{d^n}\right)\\
&=\exp\left(\frac{tn}{2}\right) \frac{1}{(n-1)!}\left(-\frac{t}{2d}\right)^{n-1} \cdot 2^{n-1}(n-1)!n^{n-2} +O\left(\frac{1}{d^n}\right)\\
&=\exp\left(\frac{tn}{2}\right) \left(-\frac{t}{d}\right)^{n-1} n^{n-2} +O\left(\frac{1}{d^n}\right).
\end{align*}
Therefore we have
\begin{align*}
\kappa_n^{(d)}(I_d(x ; t)) &= \frac{(-d)^{n-1}}{(n-1)!} \sum_{\pi\in \calP(n)}\widetilde{a}_\pi(I_d(x ; t)) \mu_n(\pi,1_n)\\
&=\exp\left(\frac{nt}{2}\right) \frac{(nt)^{n-1}}{n!} + O \left(\frac{1}{d}\right),
\end{align*}
which implies 
\begin{align*}
\lim_{d\rightarrow\infty}\kappa_n^{(d)}(I_d(x ; t)) =\exp\left(\frac{nt}{2}\right) \frac{(nt)^{n-1}}{n!}.
\end{align*}
Combining with Proposition \ref{prop:Cumulant_sigma_t} (1) and Proposition \ref{prop:converges_freecumulant}, the desired result is obtained.
\end{proof}


\begin{remark}\label{MFS}
Notice that $\lambda_t$ coincides with $D_{e^{t/2}}(\eta_t)$ where $\eta_t$ is the probability measure appeared from Sakuma and Yoshida \cite{SY}, since
\[
\kappa_n(\lambda_t)=\exp\left(\frac{nt}{2}\right) \frac{(nt)^{n-1}}{n!}=\kappa_n(D_{e^{t/2}}(\eta_t))
\]
for all $n\in \N$. 
\end{remark}

\section{Alternative proof of Kabluchko's limit theorems}

In this section, we give alternative proofs of Kabluchko's two limit theorems. As a noteworthy observation, the proofs in Section 6.1 run parallel to those presented in Section 5.

\subsection{Kabluchko's limit theorem for unitary Hermite polynomial}

Let us define $H_d(z;t)$ as a polynomial on $\C$ by setting
\[
H_d(z;t):=\sum_{k=0}^d (-1)^{k} \binom{d}{k} \exp\left(-\frac{tk(d-k)}{2d}\right) z^{d-k}, \qquad z\in \mathbb{C},\ t \ge 0.
\]
The polynomial $H_d(z;t)$ is called the {\it unitary Hermite polynomial} with parameter $t \ge 0$. By \cite[Lemma 2.1]{K22}, all zeroes of the polynomial $H_d(z;t)$ are located on the unit circle $\mathbb{T}$.
It is known as the limit polynomial of the CLT for finite free multiplicative convolution of polynomials with roots located on $\mathbb{T}$ by Mirabelli \cite[Theorem 3.16]{M}.

{
For a polynomial $p \in \monP(d)$ having only roots located on $\mathbb{T}$, there are unique $\theta_1,\dots ,\theta_d \in [-\pi,\pi)$ such that $p(z)=\prod_{k=1}^d (z-e^{i\theta_k})$.
We then define a polynomial $\phi_c(p)$ as
\[
\phi_c(p)(z):=\prod_{k=1}^d (z-e^{ic\theta_k})
\]
for $c \in (0,1)$.
}


\begin{proposition}\label{prop:CLT1}
  Let $d \ge 2$.
  Suppose $p(z) = \prod_{k=1}^d(z-e^{i\theta_k})$  {for some $\theta_1,\dots, \theta_d \in [-\pi, \pi)$. Assume that}
  \begin{align*}
    \frac{1}{d} \sum_{k=1}^d \theta_k =0 \quad \text{and} \quad \frac{1}{d} \sum_{k=1}^d \theta_k^2 = \sigma^2.
  \end{align*}
  Then
  \begin{equation*}
    \lim_{m \to \infty} \phi_{1/\sqrt{m}}(p)^{\boxtimes_d m}(z) = H_d\left(z; \frac{d\sigma^2}{d-1}\right).
  \end{equation*}
\end{proposition}


\begin{proof}
 {For the reader's convenience, we include a sketch of the proof. A parallel argument to the proof of Theorem \ref{thm:CLTpositive} demonstrates that}
  \[
    \begin{split}
      \widetilde{a}_k(\phi_{1/\sqrt{m}}(p)) &= 1 - \frac{k(d-k)}{2(d-1)m}\sigma^2 + O(m^{-\frac{3}{2}}).
    \end{split}
  \]
  Therefore, we obtain $\lim_{m \to \infty} \phi_{1/\sqrt{m}}(p)^{\boxtimes_d m}(z) = H_d(z;{d\sigma^2}/{(d-1)})$.
\end{proof}

Kabluchko \cite[Theorem 2.2]{K22} proved that the empirical root distribution of $H_d(z ;t)$ converges weakly on the unit circle to the free normal distribution $\sigma_t$ on $\mathbb{T}$ with parameter $t \ge 0$ as $d\rightarrow\infty$, where $\sigma_t$ was introduced in \cite{BV92} and studied in \cite{Biane95, Biane97, Z14, Z15}.
 {Moreover, Cébron constructed a matricial model where the moments of (unitary) matrix-valued Brownian motion at time $t$ converge to the moments of $\sigma_t$ as the matrix size tends to infinity in \cite{Cebron}.}
The free cumulants of $\sigma_t$ are known (see \cite{DG-PN}), and we give a rigorous proof of them using the Lagrange inversion theorem in Appendix A.

\begin{theorem}[See Theorem \ref{main4} (1)]\label{thm:Kabluchko}
Consider $t \ge 0$. As $d\rightarrow\infty$, we have $\mu\llbracket H_d(z;t)\rrbracket \xrightarrow{w} \sigma_t$.
\end{theorem}

\begin{proof}
A similar argument to the proof of Theorem \ref{thm:FMS} shows
\[
\kappa_n^{(d)}(H_d(z;t))=\exp\left(-\frac{nt}{2}\right) \frac{(-nt)^{n-1}}{n!} + O\left(\frac{1}{d}\right).
\]
By Propositions \ref{prop:Cumulant_sigma_t} (2) and  Proposition \ref{prop:converges_freecumulant}, we obtain the desired result.
\end{proof}

\subsection{Kabluchko's limit theorem for unitary Laguerre polynomial}

Let $L_{d,m}$ be the \textit{unitary Laguerre polynomial}, that is,
\begin{equation*}
  L_{d,m}(z)=\sum_{k=0}^d (-1)^k \binom{d}{k} \left(1-\frac{2k}{d}\right)^m z^{d-k}.
\end{equation*}%
In \cite[Theorem 2.7]{K21}, if $m/d \rightarrow t>0$ for some $t>0$ as $d,m \rightarrow\infty$, then the empirical root distribution of $L_{d,m}(z)$ converges weakly on $\mathbb{T}$ to the free unitary Poisson distribution $\Pi_t$ with parameter $t$, see \cite{BV92, K21} for details of $\Pi_t$.
We give an alternative proof of Kabluchko's limit theorem for free unitary Poisson distribution as follows.

\begin{theorem}[See Theorem \ref{main4} (2)]\label{thm:Poi}
As $d\rightarrow\infty $ with $m/d \rightarrow t$ for some $t>0$, we have
\begin{align*}
\mu\llbracket L_{d,m}\rrbracket \xrightarrow{w} \Pi_t, 
\end{align*}
equivalently
\begin{align*}
\kappa_n^{(d)}(L_{d,m})\rightarrow (-1)^{n-1}2^n e^{-2nt} \sum_{k=1}^{n-1} \frac{(-t)^k}{k!} (2n)^{k-1} \binom{n-2}{k-1} =\kappa_n(\Pi_t),
\end{align*}
where we understand $\kappa_1^{(d)}(L_{d,m})=e^{-2t}$.
\end{theorem}
\begin{proof}
Consider $m/d \rightarrow t$ for some $t>0$ as $d,m\rightarrow\infty$. By the definition of the finite free cumulants,
\begin{align*}
\kappa_n^{(d)} (L_{d,m})=\frac{(-d)^{n-1}}{(n-1)!} \sum_{\pi \in \calP(n)} \prod_{V\in \pi} \left(1-\frac{2|V|}{d}\right)^m \mu_n(\pi,1_n).
\end{align*}
We have
\begin{align*}
\prod_{V\in \pi}\left(1-\frac{2|V|}{d}\right)^m&= \exp\left( m \sum_{V\in \pi} \log \left(1-\frac{2|V|}{d}\right)\right)\\
&= \exp\left(-\frac{m}{d} \cdot d \sum_{V\in \pi} \sum_{l =1}^\infty \frac{1}{l} \frac{(2|V|)^l}{d^l}\right)\\
&=e^{-2nt}\sum_{k=0}^\infty \frac{1}{k!} \left(-\frac{m}{d}\sum_{V\in \pi} \sum_{l=2}^\infty \frac{1}{l} \frac{(2|V|)^l}{d^{l-1}}\right)^k\\
&=e^{-2nt}\sum_{k=0}^{n-1} \frac{1}{k!} \left(-\frac{m}{d}\sum_{V\in \pi} \sum_{l=2}^\infty \frac{1}{l} \frac{(2|V|)^l}{d^{l-1}}\right)^k + O(d^{-n}).
\end{align*}
Let $t_d :=m/d$ then $t_d \to t$ as $d \to \infty$ by the assumption.
Then
\begin{align*}
\kappa_n^{(d)} (L_{d,m}) &= \frac{(-d)^{n-1}}{(n-1)!}e^{-2nt} \sum_{\pi \in \calP(n)} \sum_{k=1}^{n-1} \frac{(-t_d)^k}{k!} \left(\sum_{V\in \pi} \sum_{l=2}^\infty \frac{1}{l} \frac{(2|V|)^l}{d^{l-1}}\right)^k \mu_n(\pi,1_n) + O(d^{-1})\\
&=\frac{(-d)^{n-1}}{(n-1)!}e^{-2nt} \underbrace{\sum_{\pi \in \calP(n)} \sum_{k=1}^{n-1} \frac{(-t_d)^k}{k!} \left(\sum_{V\in \pi} \sum_{l=1}^\infty \frac{1}{l+1} \frac{(2|V|)^{l+1}}{d^{l}}\right)^k \mu_n(\pi,1_n)}_{=: M_n(d)} + O(d^{-1}).
\end{align*}
For short, we write $f_{l} (x):= {(2x)^l}/{l}$, and then
\begin{align*}
M_n(d)&=\sum_{\pi \in \calP(n)}\sum_{k=1}^{n-1} \frac{(-t_d)^k}{k!} \left(\sum_{V\in \pi} \sum_{l=1}^\infty  \frac{f_{l+1}(|V|)}{d^{l}}\right)^k \mu_n(\pi,1_n)\\
&= \sum_{\pi \in \calP(n)} \sum_{k=1}^{n-1} \frac{(-t_d)^k}{k!}  \prod_{j=1}^k \left(\sum_{V\in \pi} \sum_{l_j=1}^\infty  \frac{f_{l_j+1}(|V|)}{d^{l_j}}\right) \mu_n(\pi,1_n) \\
&=\sum_{k=1}^{n-1} \frac{(-t_d)^k}{k!} \sum_{l_1,\dots, l_k=1}^\infty \frac{1}{d^{l_1+\cdots+ l_k}} \sum_{\pi \in \calP(n)} \left(\sum_{V\in \pi} f_{l_1+1}(|V|)\right)\cdots \left(\sum_{V\in \pi} f_{l_k+1}(|V|)\right)\mu_n(\pi,1_n).
\end{align*}
Recall that, for each $n\in \N$ and $l_1,\cdots, l_k \in \N$, 
$$
s_k^{(n)}[f_{l_1+1},\dots, f_{l_k+1}](0)=\sum_{\pi \in \calP(n)} \left(\sum_{V\in \pi} f_{l_1+1}(|V|)\right)\cdots \left(\sum_{V\in \pi} f_{l_k+1}(|V|)\right)\mu_n(\pi,1_n).
$$ 

By Theorem \ref{thm:Combinatorics}, if $\sum_{i=1}^k \deg (f_{l_{i}+1}) -k < n-1$, that is, $\sum_{i=1}^k l_i<n-1$ then 
$$
s_k^{(n)}[f_{l_1+1},\dots, f_{l_k+1}](0)=0,
$$
and also if $\sum_{i=1}^k l_i=n-1$ then
\begin{align*}
  s_k^{(n)}[f_{l_1+1},\dots, f_{l_k+1}](0)&=(n-1)! n^{k-1}\left(\prod_{i=1}^k (l_i+1) \cdot \frac{2^{l_i+1}}{l_i+1} \right)\\
  &=(n-1)! 2^n(2n)^{k-1}.
\end{align*}
Thus, we obtain
\begin{align*}
\kappa_n^{(d)}(L_{d,m})&=\frac{(-1)^{n-1}}{(n-1)!}e^{-2nt} \sum_{k=1}^{n-1} \frac{(-t_d)^k}{k!} \sum_{\substack{l_1,\dots, l_k=1\\ l_1+\cdots +l_k=n-1}}^\infty s_k^{(n)}[f_{l_1+1},\dots, f_{l_k+1}](0)+O(d^{-1})\\
&=(-1)^{n-1} 2^n e^{-2nt} \sum_{k=1}^{n-1} \frac{(-t_d)^k}{k!}(2n)^{k-1} \# \left\{(l_1,\dots, l_k)\in \N^k \;\middle|\; \sum_{i=1}^k l_i =n-1\right\} +O(d^{-1})\\
&=(-1)^{n-1} 2^n e^{-2nt} \sum_{k=1}^{n-1} \frac{(-t_d)^k}{k!}(2n)^{k-1}\binom{n-2}{k-1} + O(d^{-1}).
\end{align*}
Combining with Proposition \ref{prop:Poisson}, we have obtained the result by taking $d \to \infty$.
\end{proof}

\section*{Acknowledgement}

This work was supported by JSPS Open Partnership Joint Research Projects grant no. JPJSBP120209921. Moreover, Y.U. is supported by JSPS Grant-in-Aid for Scientific Research (B) 19H01791 and JSPS Grant-in-Aid for Young Scientists 22K13925.
K.F. is supported by the Hokkaido University Ambitious Doctoral Fellowship (Information Science and AI).
O.A. is supported by CONACYT Grant CB-2017-2018-A1-S-9764.

O.A. thanks Jorge Garza Vargas, Zahkar Kabluchko, Daniel Perales and Noriyoshi Sakuma for fruitful discussions. In particular, we really appreciate Jorge Garza Vargas for pointing out the relation between the combinatorial formulas of Section 3 with the problem solved in Section 6.1 and Noriyoshi Sakuma for proposing working on Theorem \ref{main3}.

The authors are grateful to anonymous referees for very helpful comments to improve the readability of the paper.



\appendix

\section{Free cumulants via the Lagrange inversion theorem}

In this appendix, we compute the free cumulants of the multiplicative free semicircular distribution $\lambda_t$, the free unitary normal distribution $\sigma_t$ and the free unitary Poisson distribution $\Pi_t$ appeared from Sections 5 and 6.
A calculation strategy is the use of the {\it Lagrange inversion theorem} (see, e.g., \cite[p. 148, Theorem A]{Comtet}). 

Recall that, for a probability measure $\mu$ (on $[0,\infty)$ or $\mathbb{T})$ with nonzero first moment, we obtain
\[
R_{\mu}(zS_\mu(z))=z
\]
on a neighborhood of $0$, where $R_\mu(z)$ is the R-transform of $\mu$ and $S_\mu(z)$ is the S-transform of $\mu$, see \cite{BV92, BV93} for details. If $f_\mu(z):=zS_\mu(z)$ is analytic on a neighborhood of $0$ and $f_\mu(0)=0$ and also $f_\mu'(0)\neq 0$, then the Lagrange inversion theorem implies that
\[
\kappa_n(\mu)= \frac{1}{n!}\lim_{z\rightarrow 0} \left(\frac{d}{dz}\right)^{n-1} \left(\frac{z}{f_\mu(z)}\right)^n.
\]

\subsection{Multiplicative free semicircular distribution $\lambda_t$ and free normal distribution $\sigma_t$}

Using the above strategy, we provide proof of the following known results.

\begin{proposition}\label{prop:Cumulant_sigma_t}
\begin{enumerate}[\rm (1)]
\item Let $\lambda_t$ be the multiplicative free semicircular distribution on $[0,\infty)$ for $t \ge 0$.
For $n\in\N$, we have
\begin{align*}
\kappa_n(\lambda_t)= \exp\left(\frac{nt}{2}\right)\frac{(nt)^{n-1}}{n!}.
\end{align*}
\item (see e.g. \cite{DG-PN}) Let $\sigma_t$ be the free normal distribution on $\mathbb{T}$ for $t \ge 0$.
For $n\in\N$, we have
\begin{align*}
\kappa_n(\sigma_t)= \exp\left(-\frac{nt}{2}\right)\frac{(-nt)^{n-1}}{n!}.
\end{align*}
\end{enumerate}
\end{proposition}

\begin{proof}
We only prove statement (1). Recall that, for any $t \ge 0$, 
\[
S_{\lambda_t}(z)=\exp\left(-t \left(z+\frac{1}{2}\right)\right),
\]
see \cite[Lemma 7.1]{BV92}. Then it is easy to verify that $f_{\lambda_t}(z):=zS_{\lambda_t}(z)$ is analytic on a neighborhood of $0$ and $f_{\lambda_t}(0)=0$ and also $f_{\lambda_t}'(0)=e^{-t/2}\neq 0$. 
The Lagrange inversion theorem implies that
\begin{align*}
\kappa_n(\lambda_t) 
&=\frac{1}{n!} \lim_{z\rightarrow0} \left(\frac{d}{dz}\right)^{n-1} \left(\frac{z}{f_{\lambda_t}(z)}\right)^n\\
&=\frac{1}{n!} \lim_{z\rightarrow0} \left(\frac{d}{dz}\right)^{n-1} \exp\left(nt\left(z+\frac{1}{2}\right)\right)\\
&=\exp\left(\frac{nt}{2}\right) \frac{(nt)^{n-1}}{n!}.
\end{align*}

Recall also that 
\[
S_{\sigma_t}(z)=\exp\left(t \left(z+\frac{1}{2}\right)\right).
\]
Then the proof of statement (2) is similar to that of (1).
\end{proof}

\subsection{Free unitary Poisson distribution $\Pi_t$}

\begin{proposition}\label{prop:Poisson}
For $n\in \N$ and $t>0$, we have
\begin{align} \label{eq:fish}
\kappa_n(\Pi_t)=(-1)^{n-1}2^n e^{-2nt} \sum_{k=1}^{n-1} \frac{(-t)^k}{k!} (2n)^{k-1} \binom{n-2}{k-1}.
\end{align}
\end{proposition}

To show this, we use the Lagrange inversion theorem again. According to \cite{K21}, we have
\[
S_{\Pi_t}(z)=\exp\left( \frac{t}{z+\frac{1}{2}}\right), \qquad t>0.
\]
One can see that, $f_{\Pi_t}(z)=zS_{\Pi_t}(z)$ is analytic on a neighborhood of $0$, and $f_{\Pi_t}(0)=0$ and also $f_{\Pi_t}'(0)=e^{2t}\neq 0$. By the Lagrange inversion theorem, we have
\begin{equation}\label{eq:cumulant_Pi_t}
\begin{split}
\kappa_n(\Pi_t)
&=\frac{1}{n!} \lim_{z\rightarrow0}\frac{d^{n-1}}{dz^{n-1}}  \left(\frac{z}{f_{\Pi_t}(z)} \right)^n\\
&=\frac{1}{n!} \lim_{z\rightarrow0}\frac{d^{n-1}}{dz^{n-1}}  \exp\left(\frac{-nt}{z+\frac{1}{2}} \right).
\end{split}
\end{equation}
To compute this, we prepare the following result derived from Fa\'{a} Di Bruno's formula.

\begin{lemma}\label{lem:FDB_Formula}
Let $u$ be an analytic function on $\C$. Then
\begin{align*}
\frac{d^{n}}{dz^{n}}e^{u(z)} =  \left( \sum_{\pi \in \calP(n)} \prod_{V\in \pi} u^{(|V|)}(z) \right) e^{u(z)}.
\end{align*}
\end{lemma}
\begin{proof}
According to Fa\'{a} Di Bruno's formula, for analytic functions $f$ and $g$, we have
\begin{align*}
\frac{d^n}{dz^n} f(g(z))=\sum_{\pi \in \calP(n)} f^{(|\pi|)} (g(z)) \cdot \prod_{V\in \pi} g^{(|V|)}(z).
\end{align*}
Taking $f(z)=e^z$ and $g(z)=u(z)$, and observing that $f^{(|\pi|)}(z)=e^z$, the desired result follows.
\end{proof}


\begin{proof}{\bf (Proposition \ref{prop:Poisson})}\ 
By \eqref{eq:cumulant_Pi_t} and Lemma \ref{lem:FDB_Formula}, we obatin
\begin{align*}
\kappa_n(\Pi_t)
&=\frac{1}{n!} (-2)^{n-1} e^{-2nt} \left(\sum_{\pi\in \calP(n-1)} (-2nt)^{|\pi|} \prod_{V\in \pi} |V|!\right)\\
&=\frac{1}{n!} (-2)^{n-1} e^{-2nt} \left(\sum_{k=1}^{n-1} (-2nt)^{k}\sum_{\substack{\pi\in \calP(n-1)\\ |\pi|=k}}  \prod_{V\in \pi} |V|!\right)\\
&=(-1)^{n-1}2^ne^{-2nt} \left(\sum_{k=1}^{n-1} (-t)^k (2n)^{k-1}\sum_{\substack{\pi\in \calP(n-1)\\ |\pi|=k}} \frac{\prod_{V\in \pi} |V|!}{(n-1)!}\right).
\end{align*}
Hence, it is enough to prove 
$$
\sum_{\substack{\pi\in \calP(n-1)\\ |\pi|=k}} \frac{\prod_{V\in \pi} |V|!}{(n-1)!} = \frac{1}{k!}\binom{n-2}{k-1}
$$
by comparing with \eqref{eq:fish}.
Note that an elementary combinatorial argument shows that
\begin{align}\label{eq:comb_choose}
\sum_{\substack{p_1,\dots, p_{n-1}\ge0\\ p_1+\cdots + (n-1)p_{n-1}=n-1 \\ p_1+\cdots+p_{n-1}=k}} \binom{k}{p_1,\dots, p_{n-1}}=\binom{n-2}{k-1}, \qquad 1\le k \le n-1.
\end{align}
Thus, since the number of partitions with $p_1$ blocks of size $1$, $p_2$ blocks of size $2$,$\dots$, $p_n$ blocks of size $n$ is equal to
\begin{align*}
 \frac{n!}{p_1! p_2! \cdots p_n! (1!)^{p_1} (2!)^{p_2}\cdots(n!)^{p_n}},
\end{align*}
we have
\begin{align*}
\sum_{\substack{\pi\in \calP(n-1)\\ |\pi|=k}} \frac{\prod_{V\in \pi} |V|!}{(n-1)!}
&=\frac{1}{k!}\sum_{\substack{p_1,\dots, p_{n-1}\ge0\\ p_1+\cdots + (n-1)p_{n-1}=n-1 \\ p_1+\cdots+p_{n-1}=k}} \binom{k}{p_1,\dots, p_{n-1}}\\
&=\frac{1}{k!}\binom{n-2}{k-1},
\end{align*}
where the last equality follows from \eqref{eq:comb_choose}.
\end{proof}




\begin{thebibliography}{99}
 \bibitem{AGVP} O. Arizmendi, J. Garza-Vargas and D. Perales, Finite Free Cumulants: Multiplicative Convolutions, Genus Expansion and Infinitesimal Distributions. arXiv:2108.08489.
  \bibitem{AP} O. Arizmendi and D. Perales, Cumulants for finite free convolution. J. Combin. Theory Ser. A \textbf{ 155} (2018), 244--266.
  \bibitem{AV} O. Arizmendi and C. Vargas, Products of free random variables and $k$-divisible non-crossing partitions. Electron. Commun. Probab. \textbf{ 17} (2012),  1--13.
  \bibitem{BV92} H. Bercovici and D. Voiculescu. L\'{e}vy-Hincin type theorems for multiplicative and additive free convolution. Pacific J. Math., \textbf{ 153}, no.2, (1992), 217--248.
  \bibitem{BV93} H. Bercovici and D. Voiculescu, Free convolution of measures with unbounded support. Indiana Univ. Math. J. \textbf{ 42} (1993), 733--773.
  \bibitem{Biane95} P. Biane. Free Brownian motion, free stochastic calculus and random matrices. In Free probability theory (Waterloo, ON, 1995), volume \textbf{ 12} of Fields Inst. Commun., Pages 1--19. Amer. Math. Soc., Providence, RI.
  \bibitem{Biane97} P. Biane. Segal-Bargmann transform, functional calculus on matrix spaces and the theory of semi-circular and circular systems. J. Funct. Anal., \textbf{ 144}, no.1, (1997), 232--286. 
  \bibitem{Cebron} G. C\'{e}bron. Matricial model for the free multiplicative convolution. Ann. Probab., \textbf{ 44}, no.4 (2016), 2427--2478.
  \bibitem{Comtet} L. Comtet. {\it Advanced combinatorics.} D. Reidel Publishing Co., Dordrecht, enlarged edition, 1974. The art of finite and infinite expansions.
  \bibitem{DG-PN} N. Demni, M. Guay-Paquet and A. Nica. Star-cumulants of free unitary Brownian motion. Adv. Appl. Math. \textbf{ 69} (2015), 1--45.
    \bibitem{FU} K. Fujie and Y. Ueda, Law of large numbers for finite free multiplicative convolution of polynomials. SIGMA \textbf{ 19} (2023), 004, 11 pages.
    \bibitem{HM} U. Haagerup and S. M\"{o}ller.
Tha law of large numbers for the free multiplicative convolution.
Operator algebra and dynamics, 157--186, Springer Proc. Math. Stat., \textbf{ 58}, Springer, Heidelberg, 2013.
  \bibitem{HLP34} G. H. Hardy, J. E. Littlewood and G. Po\'{l}ya. Inequalities. Cambridge at the university press. 1934.
  \bibitem{K21} Z. Kabluchko, Repeated differentiation and free unitary Poisson processes, arXiv:2112.14729.
  \bibitem{K22} Z. Kabluchko, Lee-Yang zeroes of the Curie-Weiss ferromagnet, unitary Hermite polynomials, and the backward heat flow, arXiv:2203.05533.
\bibitem{LS59} V. P. Leonov and A. N. Sirjaev. On a method of semi-invariants. Theor. Probability Appl., 4:319–329, 1959.
  \bibitem{M92} H. Maassen, Addition of freely independent random variables. J. Funct. Anal. \textbf{ 106} (1992), 409--438.
  \bibitem{Marcus} A. W. Marcus, Polynomial convolutions and (finite) free probability. arXiv:2108.07054.
  \bibitem{MSS22} A. W. Marcus, D. A. Spielman and N. Srivastava, Finite free convolutions of polynomials. Probab. Theory Relat. Fields \textbf{ 182} (2022), no. 3--4, 807--848.
  \bibitem{Mar66} M. Marden, Geometry of polynomials, 2nd edn. Number 3. American Mathematical Society, 1966.
  \bibitem{MS} J. A. Mingo and R. Speicher, Free probability and random matrices, Fields Institute Monographs, vol. \textbf{ 35}, Springer, New York; Fields Institute for Research in Mathematical Sciences, Toronto, ON, 2017.
  \bibitem{M} B. Mirabelli, Hermitian, Non-Hermitian and Multivariate Finite Free Probability. 2021. URL https://dataspace.princeton.edu/handle/88435/dsp019593tz21r. PhD Thesis, Princeton University.
  \bibitem{NS} A. Nica and R. Speicher, Lectures on the Combinatorics of Free Probability Theory, London Mathematical Society Lecture Note Series, vol. \textbf{ 335}, Cambridge University Press, Cambridge, 2006.
  \bibitem{SY} N. Sakuma and H. Yoshida, New limit theorems related to free multiplicative convolution. Studia Math. \textbf{ 214} (2013), no.3, 251--264. 
  \bibitem{Sz22} G. Sze\"{g}o. Bemerkungen zu einem Satz von J. H. Grace \"{u}ber die Wurzeln algebraischer Gleichungen, Math. Z. \textbf{ 13} (1), 28--55, 1922.
  \bibitem{Tho13} E. G.F. Thomas, A polarization identity for multilinear maps, Indagationes Mathematicae \textbf{ 25}, Issue 3 (2014), 468--474
  \bibitem{VDN92} D. Voiculescu, K. J. Dykema and A. Nica, Free random variables. Number 1. American Mathematical Soc., 1992.
  \bibitem{V86} D. Voiculescu, Addition of certain noncommuting random variables. J. Funct. Anal. \textbf{ 66} (1986), 323--346 .
  \bibitem{Yancey} J. Yancey, Edge-labelled trees, Pi Mu Epsilon Journal, Vol. 8, No. 2 (1985), 96--102.
  \bibitem{Z14} P. Zhong. Free Brownian motion and free convolution semigroups: multiplicative case. Pacific J. Math., \textbf{ 269}, no.1 (2014), 219--256.
\bibitem{Z15} P. Zhong. On the free convolution with a free multiplicative analogue of the normal distribution.
J. Theoret. Probab., \textbf{ 28}, no.4 (2015), 1354--1379.
\end{thebibliography}
\end{document}